\documentclass[11pt]{amsart}

\usepackage [utf8] {inputenc}
\usepackage {t1enc}
\usepackage [english] {babel}
\usepackage {calc}
\usepackage {latexsym}
\usepackage {amssymb}
\usepackage {amsmath}
\usepackage {amsthm}
\usepackage {comment}
\usepackage {hyperref}
\usepackage {enumitem}
\usepackage {hyperref}
\usepackage {xcolor}
\usepackage[margin=1in]{geometry}
\usepackage{soul}
\usepackage{cancel}
\hypersetup{colorlinks=true, linkcolor=blue, filecolor=blue, urlcolor=blue, citecolor=blue}

\newcommand{\Mod}[1]{~(\mathrm{mod}~#1)}
\newcommand{\diag}{\mathrm{diag}}
\newcommand{\disc}{\mathrm{disc}}
\newcommand{\SNF}{\mathrm{SNF}}
\newcommand{\ZZ}{\mathbb{Z}}
\newcommand{\QQ}{\mathbb{Q}}
\newcommand{\RR}{\mathbb{R}}
\newcommand{\GL}{\mathrm{GL}}

\newtheorem{theorem}{Theorem}
\newtheorem{proposition}{Proposition}[section]
\newtheorem{lemma}[proposition]{Lemma}
\newtheorem{corollary}[proposition]{Corollary}

\theoremstyle{definition}
\newtheorem*{definition}{Definition}

\theoremstyle{remark}

\newtheorem{remark}[proposition]{Remark}
\newtheorem{example}[proposition]{Example}

\title{Extensions of integral orthoregular sets and icubes}
\author{Márton Erdélyi}
\author{Péter Maga}
\author{Gergely Zábrádi}
\date{}

\address{HUN-REN Alfréd Rényi Institute of Mathematics, POB 127, Budapest H-1364, Hungary}
\email{merdelyi@math.bme.hu, magapeter@gmail.com, gergely.zabradi@ttk.elte.hu}
\address{Budapest University of Technology and Economics, Institute of Mathematics, Department of Algebra and Geometry, Egry József u. 1, Budapest H-1111, Hungary}
\email{merdelyi@math.bme.hu}
\address{Eötvös Loránd University, Institute of Mathematics, Pázmány Péter sétány 1/C, Budapest H-1117, Hungary}
\email{gergely.zabradi@ttk.elte.hu}

\begin{document}



\maketitle

\begin{abstract} In this paper, we study the question when a (rational or Gaussian) integral vector can be extended to an integral orthogonal basis consisting of vectors of equal length. We also study when a set of integral vectors has such an extension. Some necessary conditions are given which are proven to be sufficient in dimensions $3$ and $4$.
\end{abstract}

\section{Introduction}

Given an integral vector $\left(\begin{smallmatrix}x \\ y \end{smallmatrix}\right)\in\ZZ^2$, $\left(\begin{smallmatrix}-y \\ x \end{smallmatrix}\right)\in\ZZ^2$ is also integral, has the same length as the original one, and they are orthogonal to each other. This motivates the question in higher dimension: given an integral vector, can we find orthogonal integral vectors of the same length, or, most optimistically, can we complete our first vector to an orthogonal basis consisting of integral vectors of equal length?
The answer in this generality is negative, as one can readily check that if $n$ is odd and all the coordinates of $v\in\ZZ^n$ are odd, then for any $w\in\ZZ^n$ we have $|w|^2=w^Tw\equiv v^Tw \Mod{2}$ showing that if $w$ is orthogonal to $v$ of then $|w|^2$ is even, thus there is no integral vector which is orthogonal to $v$ of the same length.

The more refined question is then of course \emph{when} a given vector $v\in\ZZ^n$ can be continued to an orthogonal system consisting of vectors of the same length. This turns out to be surprisingly challenging even in low dimension. Before describing the history of the problem and our new results, we introduce the main notions following
\cite{GoswickKissMoussongSimanyi}. (All along the paper, vectors are thought of as column vectors.)
\begin{definition}[$k$-icube in $\ZZ^n$] For $1\leq k\leq n$, a matrix $(v_1|\dotsc|v_k)\in \ZZ^{n\times k}$ is called a $k$-icube (in $\ZZ$, of norm $\lambda > 0$),
if
\[
v_i^T v_j = \begin{cases} \lambda, \qquad & \text{if $i=j$}, \\ 0, & \text{if $i\neq j$.} \end{cases}
\]
In particular, a nonzero integral vector $v$ alone is a $1$-icube of norm $|v|^2$. For $1\leq\ell < k$, we say that an $\ell$-icube $A$ can be extended to a $k$-icube, if some $\ell$ columns of some $k$-icube form $A$. An icube is a $k$-icube for some $1\leq k\leq n$.
\end{definition}

The analysis of rational integral $3$-icubes has been started by Sárközy \cite{Sarkozy} using Euler matrices (see also the work of Horváth \cite{Horvath} for more elementary methods) who determined the possible norms of $3$-icubes. Later on, Goswick, Kiss, Moussong and Simányi \cite{GoswickKissMoussongSimanyi}, and Kiss and Kutas \cite{KissKutas} investigated the extendability of icubes in dimensions $3$ and $4$. These questions were also studied by Lacalle and Gatti \cite{LacalleGatti}, and Chamizo and Jiménez-Urroz \cite{ChamizoJimenezUrroz} independently, motivated by a model for discrete quantum computation \cite{GattiLacalle}. During the review process we came across a paper of Hsia \cite{Hsia}, which completely solved the extendability of icubes in dimension at most 8, and hence subsumes all the previously mentioned works (and also some of our results which we have tought to be new). However, we decided to keep the structure of the text, in paticular the introduction, because we think that it gives the reader a more complete picture of the area.

Our new findings in this area can be grouped into four topics. The first one collects some observations about the extendability of vectors ($1$-icubes) to $n$-icubes both over $\ZZ$ and $\ZZ[i]$.  The second one solves completely the question of extendability of icubes in dimensions $3$ and $4$ over the Gaussian integers (whose counterpart over the rational integers had previously been solved by others). The third one is the $2$-dimensional setup (both over the rationals and quadratic number fields) for arbitrary hermitian forms in place of the standard inner product of the $n$-space, revealing a delicate connection to divisibility relations in orders of certain quaternion algebras. Lastly, we provide an application in the sup-norm problem of automorphic forms -- which in fact had been our main motivation to investigate the extendability of Gaussian integral icubes. 

In the following these four topics are presented in more detail, separated to Sections \ref{SsecTop1}-\ref{SsecTop4}.

\subsection{Extension of vectors to $n$-icubes}\label{SsecTop1}~

For $n=2,4,8$ every integral vector can be extended to an  $n$-icube and the construction uses the algebraic properties of complex numbers, quaternions and octonions respectively. This is based on the existence of Hurwitz algebras, so we have no such constructions for other values of $n$ (see for example the first page of \cite{KissKutas}).

\begin{proposition}[{\cite[Proposition~1.3]{GoswickKissMoussongSimanyi}}, already proven in {\cite[p.233]{Sarkozy}}, attributed to Kárteszi]\label{2kp1}
If $n$ is odd, then the norm of any $n$-icube in $\ZZ^n$ is a square.
\end{proposition}

We will prove
\begin{theorem}\label{4kp2}
If $n=4k+2$ and $v\in\mathbb Z^n$ is contained in an $n$-icube, then $|v|^2\in\mathbb Z$ is the sum of two squares.
\end{theorem}

Note that this is a generalization of \cite[Proposition~18]{LacalleGatti}. 

One can synthesize Proposition \ref{2kp1} and Theorem \ref{4kp2} as follows:
\begin{corollary}[{\cite[Corollary 1]{Hsia}}]\label{dsquare}
If $v\in\mathbb Z^n$ is contained in an $n$-icube, then $|v|^2\in\mathbb Z$ is the sum of $(4,n)$ squares.
\end{corollary}

We have to remark that Hsia's paper proves this statement, but as our method is completely different, we present a new proof.

The converse is not true in general, for example the vector $(1,1, \dots, 1)^T\in\ZZ^9$ has square norm but the argument of the first paragraph shows that it cannot be extended to a $9$-icube. On the other hand $(3,0,0,\dots,0)^T\in\mathbb Z^9$ is also of norm 9, and it has an obvious extension to a $9$-icube, so one cannot expect a characterization of vectors which can be extended to an icube of full rank based solely on the norm.

However for $n=3$ we have

\begin{theorem}[{\cite[Theorem~1.4]{GoswickKissMoussongSimanyi}}, already proven in \cite{Sarkozy} ]\label{3Zv}
Any vector in $\ZZ^3$ of square norm can be extended to a $3$-icube.
\end{theorem}

The paper \cite{Hsia} proves that the converse of Corollary \ref{dsquare} holds for $n\leq 8$. A quick sketch of 

Now we turn to the Gaussian counterpart of the area. Given any complex matrix $M$ (including complex column vectors), $M^*:=\overline{M}^T$ stands for its conjugate transpose.

\begin{definition}[$k$-icube in $\ZZ[{i]}^n$] For $1\leq k\leq n$, a matrix $(v_1|\dotsc|v_k)\in \ZZ[i]^{n\times k}$ is called a $k$-icube (in $\ZZ[i]$, of norm $\lambda > 0$),
if
\[
v_i^* v_j = \begin{cases} \lambda, \qquad & \text{if $i=j$}, \\ 0, & \text{if $i\neq j$.} \end{cases}
\]
In particular, a nonzero Gaussian integral vector alone is a $1$-icube. For $1\leq\ell < k$, we say that an $\ell$-icube $A$ can be extended to a $k$-icube, if some $\ell$ columns of some $k$-icube form $A$. An icube is a $k$-icube for some $1\leq k\leq n$.
\end{definition}

We prove the $\mathbb Z[i]$-analogues of the above statements.

\begin{proposition}\label{2kp1i}
If $n$ is odd and $v\in\mathbb Z[i]^n$ is contained in an $n$-icube of norm $\lambda$, then $\lambda\in\mathbb Z$ is the sum of two rational squares.
\end{proposition}

The converse of Proposition \ref{2kp1i} is not true in general: assume that $n$ is odd, and that none of the coordinates of the vector $v\in\ZZ[i]^n$ is divisible by $1+i$. Then for any $w\in\ZZ[i]^n$ of the same length, we have
\[1\equiv n\cdot1\equiv |w|^2=w^*w\equiv v^*w \Mod{1+i}.\]
This shows that there is no integral vector orthogonal to $v$ of the same length.

Thus $(1,1,1,1,1)^T\in\mathbb Z[i]^5$ of norm $5=2^2+1^2$ cannot be extended to an icube of full rank. 

\begin{theorem}\label{3Ziv}
Assume $\lambda$ is the sum of two rational squares. Then any vector in $\ZZ[i]^3$ of norm $\lambda$ can be extended to a $3$-icube.
\end{theorem}

We also have
\begin{theorem}\label{4Ziv}
Any vector in $\ZZ[i]^4$  can be extended to a $4$-icube.
\end{theorem}

Note that while one can extend any vector $v\in\ZZ^4$ to a $4$-icube by permuting its coordinates and adding signs, over $\ZZ[i]$ we have no such a construction. The proof is based on the structure of factorizations of integral quaternions, and allows to extend any icube (not only vectors).

\subsection{Extension of icubes in $\ZZ[i]^3$ and $\ZZ[i]^4$}\label{SsecTop2}~

Now we turn to the more refined question of extending icubes instead of vectors.

\begin{theorem}[{follows from \cite[Theorem~1.1]{KissKutas}}]\label{3Z}
Any icube in $\ZZ^3$ of square norm can be extended to a $3$-icube.
\end{theorem}

\begin{theorem}[{\cite[Theorem~1.2]{KissKutas}}]\label{4Z}
Any icube in $\ZZ^4$ can be extended to a $4$-icube.
\end{theorem}

One can view that these theorems depend on the structure of the integer quaternions: In most of the earlier works over $\mathbb Z$ the proofs are directly based on the algebraic structure of the (integral) quaternions or Euler-matrices. 

In \cite{Hsia} a much stronger result is proven:

\begin{theorem}[{Reformulation of \cite[Theorem 2]{Hsia}}]\label{HsiaThm2} Let $k<n$ be integers, with $n-7\leq k$, let $d=(4,n)$ and $\lambda\in\ZZ$ be a sum of $d$ squares. Then in $\ZZ^n$ any $k$-icube of norm $\lambda$ can be extended to an $n$-icube.
\end{theorem}

\begin{remark}
The main tools for the proof are the local-global principle and the explicit classification of positive definite integral quadratic forms of discriminant 1 in dimension $n-k$.

The number 7 appearing in the theorem reflects the fact that for $n\leq 7$ there is a unique genus and there is a single class in it.

For $n-k=8$ there are two genera with class number 1, which explains why there are no general extension of icubes in $\ZZ^n$ for $n\geq9$. However, it is not clear how to decide whether an icube can be extended for $n-k>7$.
\end{remark}

\begin{corollary}
Let $n\leq8$, $k:=(4,n)$. Then any icube in $\ZZ^n$ of norm $\lambda$ satisfying $\lambda$ is a sum of $k$ squares can be extended to a $n$-icube.
\end{corollary}

We prove the $\ZZ[i]$-analogues of Theorems \ref{3Z} and \ref{4Z}. These are new results and we can not see how one could deduce them with the methods of \cite{Hsia}:

\begin{theorem}\label{3Zi}
Assume $\lambda$ is the sum of two rational squares. Then any icube in $\ZZ[i]^3$ of norm $\lambda$ can be extended to a $3$-icube.
\end{theorem}

\begin{theorem}\label{4Zi}
Any icube in $\ZZ[i]^4$  can be extended to a $4$-icube.
\end{theorem}

Our approach is a novel one (somewhat similar to that of \cite{Horvath}), which makes the connection of the icubes over $\ZZ[i]$ and the number theory of quaternions more apparent. We also reprove Theorems \ref{3Z} and \ref{4Z} by relating icubes over $\ZZ$ and number theory of Gaussian integers.

We conclude this Section with giving some high-dimensional (counter)examples:

\begin{example}
Let $a_1=(3,0,0,\dots,0)$, $a_2=(0,1,1,\dots,1)\in \mathbb Z^{10}$. Then $a_1$ and $a_2$ can be extended to a 10-icube, but $A=(a_1|a_2)$ cannot be extended to a 3-icube. Indeed, if $b_j:=(0,\dotsc,0,3,0,\dotsc,0)^T$ with the single $3$ standing in the $j$th entry gives the $10$-icube $(a_1|b_2|\dotsc|b_{10})$, while \[\begin{pmatrix}
0 & 1 & 1 & 1 & 1 & 1 & 1 & 1 & 1 & 1 \\
1 & -2 & -1 & 0 & 0 & 0 & 0 & 1 & 1 & 1 \\
1 & -1 & 1 & 0 & 0 & 1 & 1 & -2 & 0 & 0 \\
1 & 0 & 0 & -1 & 1 & 0 & 1 & 1 & -2 & 0 \\
1 & 0 & 0 & 1 & -1 & 1 & 0 & 1 & 0 & -2 \\
1 & 0 & 1 & 0 & 1 & -2 & 0 & 0 & 1 & -1 \\
1 & 0 & 1 & 1 & 0 & 0 & -2 & 0 & -1 & 1 \\
1 & 1 & -2 & 1 & 1 & 0 & 0 & -1 & 0 & 0 \\
1 & 1 & 0 & -2 & 0 & 1 & -1 & 0 & 1 & 0 \\
1 & 1 & 0 & 0 & -2 & -1 & 1 & 0 & 0 & 1 \\
\end{pmatrix}\]
is an extension of $a_2$ to a $10$-icube.
\end{example}

However, if $(a_1|a_2|a_3)$ were a $3$-icube, then the orthogonality of $a_1$ and $a_3$ would give $(a_3)_1=0$, and then the orhogonality of $a_2$ and $a_3$ together with the norm condition on $a_3$ would imply that
\[
\sum_{j=2}^{10} (a_3)_j = 0,\qquad \sum_{j=2}^{10} (a_3)_j^2 = 9,
\]
which viewed modulo $2$ would yield the contradiction that $\sum_{j=2}^{10}(a_3)_j$ is simultaneously even and odd.

We also have that the 10-icube $(1,1,1,1,\dots,1,1)^T$, $(3,-3,0,0,\dots,0,0)^T$, $(0,0,3,-3,\dots,0,0)^T$, $\dots$, $(0,0,0,0,\dots,3,-3)^T$ in $\ZZ^{18}$ cannot be extended to a 11-icube. Indeed any vector orthogonal to the last 9 has the form $(\alpha_1,\alpha_1,\alpha_2,\alpha_2,\dots,\alpha_9,\alpha_9)^T$ and this reduces to find vectors orthogonal to $(1,1,\dots,1)\in\ZZ^9$.

In the same manner one can construct a 28-icube in $\ZZ^{36}$ which cannot be extended to a 29-icube: 
$(1,1,1,1,1,\dots,1,1,1,1)^T$,
$(3,3,-3,-3,0,\dots,0,0,0,0)^T$,
$(3,-3,3,-3,0,\dots,0,0,0,0)^T$,\break
$(3,-3,-3,3,0,\dots,0,0,0,0)^T$,
$\dots$,
$(0,0,0,0,0,\dots,3,3,-3,-3)^T$,
$(0,0,0,0,0,\dots,3,-3,3,-3)^T$,\break
$(0,0,0,0,0,\dots,3,-3,-3,3)^T$.
The vectors orthogonal to the last 27 are in the form \[(\alpha_1,\alpha_1,\alpha_1,\alpha_1,\alpha_2,\alpha_2,\dots,\alpha_9,\alpha_9,\alpha_9,\alpha_9)^T.\] This provides a counterexample for Conjectures 3 and 4 in \cite{LacalleGatti}.

\subsection{$Q$-orthoregular bases}\label{SsecTop3}~

In this Section let $K=\QQ$ or an imaginary quadratic field and $R:=\mathcal O_K$ be the ring of integers in $K$.

\begin{definition}[integral $Q$-orthoregular bases in $R^n$] 
Let $Q$ be an hermitian form corresponding to the matrix $M=M^*\in R^{n\times n}$. 
A matrix $(v_1|\dotsc|v_n)\in R^{n\times n}$ is called an integral $Q$-orthoregular basis (in $R^n$, of norm $0<\lambda\in\mathbb{Z}$),
if
\[
v_i^*M v_j = \begin{cases} \lambda, \qquad & \text{if $i=j$}, \\ 0, & \text{if $i\neq j$.} \end{cases}
\]
\end{definition}

If $Q$ is the standard inner product on $K^n$, then an integral $Q$-orthoregular basis is exactly an $n$-icube. If there exists an integral $Q$-orthoregular basis, then $Q$ is clearly positive definite.

The motivation of the definition of integral $Q$-orthoregular bases comes from the question of extending a vector to a 3-icube in $R^3$. After fixing $a_1\in R^3$ we investigate the hermitian form $Q$ of discriminant $|a_1|^2$, which arises by restricting the norm to the $R$-module
\begin{equation}\label{Lambdalattice}
\Lambda:=\{w\in R^3 \mid a_1^*w=0\}.
\end{equation}
Finding a 3-icube containing a fixed vector $a_1\in R^3$ then reduces to finding an integral $Q$-orthoregular basis of norm $|a_1|^2$ in $\Lambda$. 

Let $Q$ be a binary hermitian form corresponding to the matrix $M\in R^{2\times 2}$. Note that if $A=(v_1|v_2)$ is an integral $Q$-orthoregular basis of norm $\lambda$, then $A^*MA=\left(\begin{smallmatrix} \lambda & 0 \\ 0 & \lambda \end{smallmatrix}\right)$, hence $\det(M)=|y|^2$ for $y=\lambda/\det(A)$. Consequently we introduce the terminology that a number $\mu\in\QQ$ is an \emph{absolute square in $K$} if there is a $y\in K$ such that $|y|^2=\mu$. Note that in general $|a_1|^2$ of the previous paragraph is not an absolute square in $K$ as $a_1$ is a vector, not an element of $K$.

For icubes we will need $K=\QQ$ and $K=\QQ(i)$. In particular, $\mu\in\ZZ$ is an absolute square in $\QQ$ if and only if $\mu$ is a square integer and $\mu\in\ZZ$ is an absolute square in $\QQ(i)$ if it is a sum of two square integers. Indeed, by Propositions \ref{2kp1} and \ref{2kp1i} the vector $a_1\in R^3$ can be extended to a 3-icube if and only if $|a_1|^2$ is an absolute square in $K$.

Thus for now assume $\Delta:=\det(M)$ is an absolute square in $K$ and note that a slight modification of the definition of integral $Q$-orthoregular bases allows us to tackle the general case, see Section \ref{QuaternionSection}.

Assume also that $\lambda$ is a multiple of $\Delta$, say $\lambda=\nu\Delta$ for some $\nu\in\ZZ$.  Then the integral $Q$-orthoregular bases of norm $\lambda$ have a rigid structure, as we shall prove in Section \ref{QuaternionSection}:

\begin{proposition}\label{perpprop}~
If $(a_1,a_2)\in R^2\times R^2$ is a $Q$-orthoregular basis of norm $\lambda=\nu\Delta$, then $a_2=(Ma_1)_\perp/\delta$ for some $\delta\in \frac1\nu R$, $|\delta|^2=\Delta$, where
\begin{equation}\label{perpdef}
\begin{pmatrix}x\\y \end{pmatrix}_\perp=\begin{pmatrix}-\overline y\\ \overline x\end{pmatrix}.
\end{equation}
\end{proposition}

This motivates the following
\begin{definition}
The $Q$-orthoregular basis $(a_1,a_2)\in R^2\times R^2$ is of type $\delta$ if $a_2=(Ma_1)_\perp/\delta$.
\end{definition}

We state here the main result about $Q$-orthoregular bases for general $K$:

\begin{theorem}\label{OrthoregbasisthmK}
Let $K$ be an imaginary quadratic field of class number one with ring of integers $R=\mathcal O_K$. Let $Q$ be a binary hermitian form corresponding to a matrix $M$ with $\Delta=\det(M)$ absolute square in $K$. Then any $a_1\in R^2$ with $\Delta|Q(a_1)$ can be extended to an integral $Q$-orthoregular basis.
\end{theorem}

The main tool for this result is an explicit relation between certain integral $Q$-orthoregular bases related to integral hermitian forms $Q$ over quadratic imaginary fields $K$ and factorizations of specific form in a quaternion order $S\subset\mathbb H$ of Hamiltonian quaternions. (In the case $K=\QQ$ we get an explicit relation between $Q$-orthoregular bases related to integral quadratic forms and certain factorizations in a quadratic order.)

The connection of hermitian forms and quaternion algebras goes at least back to Latimer (\cite{Latimer35} and \cite{Latimer36}) and since then many others studied it (see for example \cite[Section~2]{Shimura} and \cite{SavinZhao}). However, for orthoregular bases we need factorizations of explicit elements rather than that of ideals, so in Section \ref{QuaternionSection} we prove a version of this correspondence which is more suitable for our purpose.

Finally turn to the cases $K=\QQ$ or $K=\QQ(i)$. Then the quadratic or quaternion order corresponding to $R$ is

\[S=\begin{cases}\text{the ring of Gaussian integers}, & \text{ if } R=\mathbb Z\\ \text{the ring of Lipschitz integral quaternions}, & \text{ if }R=\mathbb Z[i].\end{cases}\]

A careful understanding of the arithmetic of these rings yields the following

\begin{proposition}\label{Orthoregbasisthm} 
Let $Q$ and $\lambda=\nu\Delta$ as above and $\delta\in\frac{1}{\nu} R$ of norm $\Delta$. Then we have
\begin{enumerate}
\item If $R=\mathbb Z$, then there exists an integral $Q$-orthoregular basis of type $\delta$ and norm $\lambda$ if and only if $\alpha\nu$ is a sum of two squares, where $\alpha=Q((1,0)^T)$.
\item If $R=\mathbb Z[i]$, then there always exists an integral $Q$-orthoregular basis of type $\delta$ and norm $\lambda$.
\item In both cases any vector $a_1\in R^2$ with $Q(a_1)=\lambda$
can be extended to an integral $Q$-orthoregular basis.
\end{enumerate}
\end{proposition}

\begin{remark}
The appearance of $\alpha$ in (1) might seem surprising. In fact $\alpha$ equals the upper left enry of the matrix $M$, and to see why it is a natural condition, we refer to Lemma \ref{squaredetlem}.
\end{remark}

This directly proves Theorems \ref{3Z} and \ref{3Zi}, see Section \ref{Dim3SSection}. For Theorems \ref{4Z} and \ref{4Zi} we have to examine the hermitian form which arises by restricting a norm to certain $R$-modules of rank 2.

\subsection{An application to automorphic forms}\label{SsecTop4}~

We also present an application from the analytic theory of automorphic forms. Assume $G$ is a reductive group defined over a number field $K$ such that $G(K_{\infty})$ is not compact for at least one archimedean place $K_{\infty}$ of $K$. Consider then the quotient space
\[
X:=G(K)\backslash G(\mathbb{A}),
\]
where $\mathbb{A}$ stands for the adele ring of $K$, and assume that $X$ has finite volume with respect to the $G(\mathbb{A})$-invariant measure (regarding the theory, this is quite natural to do). The sup-norm problem of automorphic forms asks for pointwise bounds on smooth, $L^2$-normalized functions $\phi:X\to\mathbb{C}$ which generate irredubible $G(\mathbb{A})$-representations under the right action of $G(\mathbb{A})$. We further assume that $\phi$ is spherical, i.e.~it is fixed under the action of a maximal compact subgroup of $G(\mathbb{A})$. The baseline bound is due to Sarnak \cite{Sarnak2004L}, and it says the following. For any compact set $\Omega\subseteq X$, we have that\footnote{Here and below, we follow the notation of Vinogradov: $A\ll_D B$ means that $|A|\leq CB$ for some constant $C$ depending only on $D$.}
\[
\|\phi|_{\Omega}\|_{\infty} \ll_{\Omega} \lambda_{\phi}^{\frac{\dim-\mathrm{rk}}{4}}
\]
holds for all $\phi$, where $\lambda_{\phi}$ is the Laplace eigenvalue of $\phi$, while $\dim$ and $\mathrm{rk}$ stand for the dimension and the rank, respectively, of the real symmetric space corresponding to $X$. This general bound relies only on the analytic properties of $\phi$, in particular, it holds for eigenfunctions on more general symmetric spaces, and the exponent $\frac{\dim-\mathrm{rk}}{4}$ is known to be sharp in the broad generality. The arithmetic sup-norm problem asks if it can be improved to $\frac{\dim-\mathrm{rk}}{4}-\delta(G)$ with some $\delta(G)>0$ depending only on $G$ in the arithmetic situation presented here. For $G=\mathrm{PGL}_n$ (a prominent class of groups in the theory), the answer is positive, see \cite{BlomerMagaSelecta} for $K=\QQ$ and \cite{MagaZabradi} for $K=\QQ(i)$. More general classes of groups are treated in the unpublished preprint of Marshall \cite{Marshall}, in particular, all real quasi-split simple groups except for those being isogenous to $\mathrm{SU}_{n,n-1}$, $n \geq 2$. This exceptional class seems to be untreatable with the current technology, and our investigations on icubes give a quantitative explanation to this obstruction, at least for some small values of $n$. We describe this phenomenon briefly, noting that in fact, the sup-norm problem of automorphic forms on $\mathrm{SU}_{n,n-1}$ had been the original motivation to our research on icubes.

The method towards all known results in the sup-norm problem (at least in high rank) starts with applying (some version of) the amplification method as pioneered by Iwaniec and Sarnak \cite{IwaniecSarnak} and estimating the so-called spherical transform. A careful analysis leads then to the heart of the matter: a delicate matrix-counting problem whose exact form depends a lot on what $G$ and $K$ exactly are.

For the group $G=\mathrm{SU}_{n,n-1}$ over $K=\QQ$, this matrix-counting problem is the following (after some simplifications), see e.g. \cite{BlomerMagaIMRN} for more details, which describes the setup of the analogous matrix-counting problem for $\mathrm{PGL}_n$. Assume that $\ell_1,\ell_2$ are Gaussian primes lying above split rational primes such that $|\ell_1|^2,|\ell_2|^2$ are about size $L$, a large parameter, and consider
\[
S_n(\ell_1,\ell_2):=\left\{\gamma\in \mathrm{GL}_n(\ZZ[i])\ \diag(1,\ell_1\ell_2,\dotsc,\ell_1\ell_2,|\ell_1\ell_2|^2)\ \mathrm{GL}_n(\ZZ[i]) \mid \gamma^* \gamma= |\ell_1\ell_2|^2 \cdot I \right\},
\]
where $I$ stands for the identity matrix.
Note that the set $S_n(\ell_1,\ell_2)$ is exactly the set of $n$-icubes with Smith normal form \[\diag(1,\ell_1\ell_2,\dots,\ell_1\ell_2,|\ell_1\ell_2|^2).\] To solve the sup-norm problem, it would be sufficient to prove that
\[
\# S_n(\ell_1,\ell_2) \ll_n L^{2(n-1)-\eta}
\]
holds for some $\eta>0$ (which might depend on $n$ but nothing else).

Unfortunately, this is not the case even for $n=2$. One can easily see that
\[
\# S_2(\ell_1,\ell_2) = \#\{(a,b,c,d) \in \ZZ^4 \mid a^2+b^2+c^2+d^2=|\ell_1\ell_2|^2,\ \gcd(a^2+b^2,|\ell_1\ell_2|^2)=1 \},
\]
since for any such $4$-tuple on the right-hand side, we can take
\begin{equation}\label{eq:special-hecke-returns}
\gamma:=\begin{pmatrix}
a+bi & -c+di \\ c+di & a-bi
\end{pmatrix}\in S_2(\ell_1,\ell_2).
\end{equation}
We know that a large odd number of size $L^2$ is represented as the sum of four squares in more than $L^2$ ways and asymptotically $100\%$ of the representations satisfy the coprimality condition $\gcd(a^2+b^2,|\ell_1\ell_2|^2)=1$. To sum up, $\#S_2(\ell_1,\ell_2)\gg L^2$.

For $n\geq 3$, we follow the same procedure: we make the first column $a_1$ of $\gamma$ with the only conditions that its entries are coprime and that $|a_1|^2=|\ell_1\ell_2|^2$, and then we complete the matrix. One can easily see that there are $\gg_n L^{2(n-1)}$ choices for $a_1$. For $n=3$, Proposition~\ref{3Zip} below shows that $a_1$ can be completed to a matrix considered in $S_3(\ell_1,\ell_2)$. For $n=4$, we first complete $a_1$ to $(a_1|a_2)$ as it is given in \eqref{{eq:1-icube-to-2-icube-in-Z[i]4}}. It is straightforward to check that for almost all choices of $a_1$, the condition imposed in Proposition~\ref{4Zip} below holds, hence it can be completed to a matrix counted in $S_4(\ell_1,\ell_2)$. As a result,
\[
\#S_n(\ell_1,\ell_2)\gg_n L^{2(n-1)},\qquad n=2,3,4.
\]

That is, for $n=2,3,4$, there is a high number of matrices (a high number of so-called \emph{Hecke returns}), which shows that the current technology cannot work for the groups $\mathrm{SU}_{2,1}$, $\mathrm{SU}_{3,2}$, $\mathrm{SU}_{4,3}$. For $n=2$, this was clear and classical from the four squares theorem, but the special shape of matrices (recall \eqref{eq:special-hecke-returns}) suggested that the high number of Hecke returns might be relying on the existence of a low-dimensional special algebra, and the counting problem can be sufficiently solved for larger values of $n$. Based on our findings, we think this is quite unlikely, and additional ideas and/or a profound, new approach are needed to solve the sup-norm problem for any $\mathrm{SU}_{n,n-1}$.

We introduce the following notations. First of all, for any Gaussian integral matrix $A$, we denote by $d_k(A)$ its \emph{$k$th determinantal divisor}, the greatest common divisor of its $k\times k$ subdeterminants, note that it is well-defined only up to unit multiples. For an $m\times n$ matrix $A$, we write $\SNF(A)=\diag(\alpha_1,\dotsc,\alpha_{\min(m,n)})$ for its \emph{Smith normal form}, if
\[
A\in \mathrm{GL}_m(\ZZ[i])\ \underbrace{\diag(\alpha_1,\dotsc,\alpha_{\min(m,n)})}_{\in\ZZ[i]^{m\times n}}\ \mathrm{GL}_n(\ZZ[i]), \qquad \alpha_1\mid \dotsc \mid \alpha_{\min(m,n)},\quad \Re\alpha_j> 0, \Im \alpha_j\geq 0
\]
Record also the equation
\[
\prod_{j=1}^k \alpha_j = d_k(A), \qquad k\leq \min(m,n).
\]
which is meant by an appropriate choice of the unit multiples in $d_k$ (or on the level of generated principal ideals). A matrix $A$ is called \emph{primitive}, if $d_1(A)=1$, i.e.~its entries have no nontrivial common divisor. All these notions and terminologies might be applied to vectors.

Let $R=\mathbb Z$ or $\mathbb Z[i]$. We will prove the following statements:
\begin{lemma}\label{SNFlemma}
Let $A=(a_1|a_2|\dots|a_n)\in R^{n\times n}$ be an icube of norm $\lambda$ and with Smith normal form $\SNF(A)=\diag(\alpha_1,\alpha_2,\dots,\alpha_n)$. Then for all $1\leq j\leq n$ we have $\overline\alpha_j\alpha_{n+1-j}=\lambda$.
\end{lemma}

\begin{proposition}\label{3Zip}
If $a_1\in R^3$ is primitive such that $|a_1|^2=\Delta$ is a norm in $R$, then for any $\alpha_2\in R$, $|\alpha_2|^2=\Delta$ there exists an extension to an icube $A(\alpha_2)=(a_1|a_2|a_3)$ for which $\SNF(A(\alpha_2))=\diag(1,\alpha_2,\Delta)$.
\end{proposition}

\begin{remark}
If $d_1(a_1)=1$ and $Q$ is the hermitian form which is obtained by restricting the norm to the lattice $\Lambda$ as in (\ref{Lambdalattice}), then an orthoregular basis of norm $\Delta$ and type $\delta$ yields the icube $A(\overline\delta)$.
\end{remark}

\begin{proposition}\label{4Zip}
If $A_0=(a_1|a_2)\in\mathbb Z[i]^{4\times 2}$ is a 2-icube of norm $\lambda$ such that $a_1$ is primitive and $(d_2(A_0),\overline{d_2(A_0)})=1$ (that is all prime divisors of $d_2(A)$ have prime norm $p=4k+1\in\mathbb Z$ and are pairwise non-conjugate), then for all $\alpha_2|d_2(A)$ there is an extension $A(\alpha_2)=(a_1|a_2|a_3|a_4)$ to a $4$-icube such that $\SNF(A(\alpha_2))=\diag(1,\alpha_2,\lambda/\overline{\alpha_2},\lambda)$.
\end{proposition}

\subsection{The structure of the paper.} In Section \ref{Sec1ton} we prove the results about necessary conditions for extending a vector to an $n$-icube (namely Theorem \ref{4kp2} and Proposition \ref{2kp1i}). Also the proof of Lemma \ref{SNFlemma} fits there the best.

In Section \ref{QuaternionSection} we formulate and prove the relation between integral $Q$-orthoregular bases and quaternions hinted in Section \ref{SsecTop3} and obtain the results stated there (Propositions \ref{perpprop} and \ref{Orthoregbasisthm}, Theorem \ref{OrthoregbasisthmK}).

Then the proofs of the remaining statements are completed -- these can be summarized as sufficient conditions for extension of icubes: Section \ref{Dim3SSection} treats the case $n=3$ (Theorems \ref{3Ziv}, \ref{3Z} and \ref{3Zi}), in Section \ref{Dim4SSection} the case $n=4$ is done (Theorems \ref{4Ziv}, \ref{4Z} and \ref{4Zi}) and finally in Section \ref{SNFsection} the application for automorphic forms is proven (Propositions \ref{3Zip} and \ref{4Zip}).

\subsection*{Acknowledgments} We are grateful to Péter E. Frenkel, Gergely Harcos, Emil W. Kiss, Péter Kutas and John Voight for helpful discussions. Some of the ideas and even the name ''orthoregular'' originated from the unpublished paper \cite{KobayashiNakayama} of Masanori Kobayashi and Chikara Nakayama.

The research towards this paper was supported by the MTA–RI Lendület ``Momentum'' Analytic Number Theory and Representation Theory Research Group (M.\ E., P.\ M., G.\ Z.), by the NKFIH (National Research, Development and Innovation Office) Grants FK-135218 (M.\ E., P.\ M.) and K-135885 (M.\ E., G.\ Z.), KKP-133819 (P.\ M.), and by the János Bolyai Scholarship of the Hungarian Academy of Sciences (G.\ Z.).

\section{Extension of vectors to \texorpdfstring{$n$}{n}-icubes }\label{Sec1ton}

We prove Propositions~\ref{2kp1}~and~\ref{2kp1i} in the following more general form. Let $K\leq\mathbb{C}$ be a number field stable under complex conjugation, i.e.~for  $\alpha\in K$ we have $\overline{\alpha}\in K$. Hence $K_0:=K\cap\mathbb R$ is a subfield of index at most $2$. Let us denote the ring of integers in $K$ by $\mathcal O_K$ and in $K_0$ by $\mathcal O_{K_0}:=K_0\cap \mathcal{O}_K$. Call $A\in \mathcal{O}_K^{n\times n}$ an icube of norm $\lambda$ if $A^*A=\lambda I$.

We first prove a preparatory lemma.
\begin{lemma}\label{quotientisnorm}
Assume $K$ has class number $1$. If $a,b\in\mathcal{O}_{K_0}$ satisfy that $a\mid b$ and that $a=\alpha\overline\alpha$ and $b=\beta\overline\beta$ for some $\alpha,\beta\in \mathcal{O}_K$, then there exists an $\alpha_1\in \mathcal{O}_K$ dividing $\beta$ such that $\alpha\overline{\alpha}=\alpha_1\overline{\alpha_1}$. In particular, we have $b/a=\gamma\overline\gamma$ with $\gamma=\beta/\alpha_1\in \mathcal{O}_K$.
\end{lemma}
\begin{proof}
Since $K$ has class number $1$, $\mathcal{O}_K$ has unique factorization. We proceed by induction on the number of primes $\pi\in \mathcal{O}_K$ such that $\nu_\pi(\alpha)>\nu_\pi(\beta)$ where $\nu_\pi$ stands for the $\pi$-adic valuation. In case there is no such prime we have $\alpha\mid \beta$ and $b/a=\gamma\overline{\gamma}$ with $\gamma:=\beta/\alpha\in\mathcal{O}_K$. 

Let $\pi\in\mathcal{O}_K$ be a prime such that $\overline\pi$ is a unit multiple of $\pi$ and put $r:=\nu_\pi(\alpha)$. Then we have \[\pi^{2r}\mid \pi^r\overline{\pi}^r\mid \alpha\overline\alpha=a\mid b=\beta\overline\beta,\] so we obtain $\pi^r\mid \beta$ showing $\nu_\pi(\alpha)\leq \nu_\pi(\beta)$.

Now assume that $\overline\pi$ is not a unit multiple of $\pi$. This can only happen if $K\neq K_0$. We put $r:=\nu_\pi(\alpha)$, $r':=\nu_\pi(\beta)$, $s:=\nu_{\overline{\pi}}(\alpha)=\nu_\pi(\overline{\alpha})$, $s':=\nu_{\overline{\pi}}(\beta)=\nu_\pi(\overline{\beta})$ and note
\[r+s=\nu_\pi(\alpha\overline{\alpha})\leq \nu_\pi(\beta\overline{\beta})=r'+s'.\]
Now if $r'<r$ then we define
\[\alpha':=\alpha\cdot\frac{\overline{\pi}^{r-r'}}{\pi^{r-r'}}\]
so that $\alpha\overline{\alpha}=\alpha'\overline{\alpha'}$, $\nu_\pi(\alpha')=r'=\nu_\pi(\beta)$, and $\nu_{\overline{\pi}}(\alpha')=s+r-r'\leq s'=\nu_{\overline{\pi}}(\beta)$. The statement follows by induction as neither $\pi$, nor $\overline{\pi}$, nor any new primes appear in the denominator of $\beta/\alpha'$.
\end{proof}

\begin{proposition}\label{2kp1general}
Let $n$ be odd and assume $K$ has class number $1$. If there exists an $n$-icube of norm $\lambda\mathcal{O}_{K_0}$, then there exists $r\in \mathcal{O}_K$ such that $\lambda=r\overline r$.
\end{proposition}

\begin{proof}
Assume $A\in \mathcal{O}_K^{n\times n}$ is an $n$-icube of norm $\lambda$. We compute
\[\lambda^n=\det(\lambda I) = \det(A^*A) = \det(A)\cdot \overline{\det(A)},\]
while
\[\lambda^{n-1}=\lambda^{\frac{n-1}{2}}\cdot \overline{\lambda^{\frac{n-1}{2}}}.\]

We conclude that both $\lambda^n$ and $\lambda^{n-1}\in \mathcal{O}_{K_0}$ are in the form $y\overline y$ for some $y\in \mathcal{O}_K$. The statement then follows from Lemma~\ref{quotientisnorm}.
\end{proof}

\begin{proof}[Proof of Propositions~\ref{2kp1}~and~\ref{2kp1i}] Applying Proposition~\ref{2kp1general} to $K=\QQ$ and $K=\QQ(i)$ gives Proposition~\ref{2kp1} and Proposition~\ref{2kp1i}, respectively.
\end{proof}

\begin{remark} The assumption that $K$ has class number $1$ cannot be removed 
from the statement of Lemma \ref{quotientisnorm} in general. The simplest counterexample we managed to find is $K=\QQ(\sqrt{-23})$, $a=9$, $b=27$. Then $b/a=3$ is not a norm of an element in $\mathcal{O}_K$, even though $a,b$ are.
\end{remark}

Let us continue with
\begin{proof}[Proof of Lemma \ref{SNFlemma}] Assume that
\[
A=S\ \diag(\alpha_1,\dotsc,\alpha_n)\ T,
\]
with $\alpha_1,\dotsc,\alpha_n\in R$ satisfying $\alpha_1\mid \dotsc\mid \alpha_n$, and $S,T\in\GL_n(R)$. By assumption, $A^*A=\lambda I$, hence we have
\[
A^*=\lambda A^{-1} = T^{-1}\ \lambda\ \diag(\alpha_1^{-1},\dotsc,\alpha_n^{-1})\ S^{-1}.
\]
Setting $W=W^{-1}$ for the permutation matrix with $1$'s in the antidiagonal, we can rewrite the previous equation as
\[
A^*=T^{-1}\ W\ W\ \lambda\ \diag(\alpha_1^{-1},\dotsc,\alpha_n^{-1})\ W\ W\ S^{-1},
\]
or equivalently,
\[
W\ T\ A^*S\ W=\diag(\lambda/\alpha_n,\dotsc,\lambda/\alpha_1).
\]
We have that $WT,SW\in\GL_n(R)$, hence the diagonal matrix on the right-hand side is a Smith normal form of $A^*\in \GL_n(R)\diag(\overline\alpha_1,\dotsc,\overline\alpha_n)\GL_n(R)$. Then the uniqueness of the Smith normal form yields that $\overline\alpha_j \alpha_{n+1-j} = \lambda \cdot \text{unit}$ for each $1\leq j\leq n$, and the $\text{unit}$ in question has to be $1$ by the choice $\Re\alpha_j>0$, $\Im\alpha_j\geq 0$.
\end{proof}

\begin{proof}[Proof of Theorem \ref{4kp2}] First we fix some notations. Let $n=4k+2$, $v\in\ZZ^n$ be contained in the $n$-icube $A\in\ZZ^{n\times n}$, and let $S,T\in\GL_n(\ZZ)$ satisfy that
\[
A=S\ \diag(\alpha_1,\dotsc,\alpha_n)\ T
\]
with $\alpha_1\mid \dotsc\mid \alpha_n$ and $\alpha_1,\dotsc,\alpha_n>0$.

For the sake of contradiction, assume that $|v|^2$ is not the sum of two squares. Then choose a prime number $p\equiv 3\bmod 4$ such that $\nu_p(|v|^2)=2\ell+1$ with some integer $\ell\geq 0$, where $\nu_p$ stands for the $p$-adic valuation. Write $S$ in the $(2k+1)\times (2k+1)$ block form
\[
S=\begin{pmatrix}
S_{11} & S_{12} \\ S_{21} & S_{22}
\end{pmatrix}.
\]
Since $\det S=\pm 1$, $S$ is of full rank modulo $p$, hence we may assume that $p\nmid \det S_{11}$ by permuting its rows (accordingly the rows of $A$ and the coordinates of $v$).

By the assumption that $A$ is an $n$-icube of norm $|v|^2$, we have that
\[
\begin{split}
|v|^2 I &= A^T A = (S\ \diag(\alpha_1,\dotsc,\alpha_n)\ T)^T (S\ \diag(\alpha_1,\dotsc,\alpha_n)\ T) \\ &= T^T\ \diag(\alpha_1,\dotsc,\alpha_n)\ S^T S\  \diag(\alpha_1,\dotsc,\alpha_n)\ T\equiv 0 \bmod p^{2\ell+1}.
\end{split}
\]
By $\det T=\pm 1$, this implies that
\[
\diag(\alpha_1,\dotsc,\alpha_n)\ S^T S\  \diag(\alpha_1,\dotsc,\alpha_n) \equiv 0 \bmod p^{2\ell+1},
\]
in particular, for any $1\leq i,j\leq n$,
\[
\alpha_i (S^T S)_{i,j}\alpha_j \equiv 0 \bmod p^{2\ell+1}.
\]
For any $1\leq j\leq 2k+1$, we combine $\nu_p(\alpha_j)+\nu_p(\alpha_{n+1-j})=2\ell+1$ from Lemma~\ref{SNFlemma} with $\nu_p(\alpha_j)\leq \nu_p(\alpha_{n+1-j})$ to see that $\nu_p(\alpha_j)\leq \ell$. Together with the previous display (letting $i$ and $j$ simultaneously run from $1$ to $2k+1$), this implies that the top-left $(2k+1)\times (2k+1)$ block of $S^T S$ is entrywise divisible by $p$, which, expressed in the block form, reads
\[
S_{11}^TS_{11}+S_{21}^TS_{21} \equiv 0 \Mod p.
\]
Taking determinants, we finally arrive at
\[
(\det S_{11})^2+(\det S_{21})^2 \equiv 0 \Mod p,
\]
which contradicts the joint choice of $p$ (being $3$ modulo $4$) and $S_{11}$ (being invertible modulo $p$).
\end{proof}

\section{Orthoregular bases and factorizations in quaternion orders}\label{QuaternionSection}

For this section let $K$ be either an imaginary quadratic field or $\mathbb Q$ and put $R:=\mathcal{O}_K$ for the ring of integers in $K$. 
Let
\[
M:=\begin{pmatrix} \alpha & \beta \\ \overline\beta & \gamma \end{pmatrix}\in K^{2\times 2}
\]
be a positive definite self-adjoint matrix (in particular, $\alpha,\gamma\in\QQ$), and put
\[
Q(v):=v^*Mv, \qquad v\in K^2
\]
for the corresponding hermitian form.

For $0<\mu\in\QQ$ write 
\begin{equation}\label{normfree}
\mu=\Delta\varepsilon,
\end{equation}
where $\Delta\in\QQ$ is an absolute square in $K$ and $1$ is the single divisor of $\varepsilon\in\mathbb Z$ which is an absolute square.

In particular, if $\mu\in\ZZ$ then $\varepsilon$ is the squarefree part of $\mu$ for $K=\QQ$ and $\varepsilon$ is the product of those rational prime divisors of $\mu$ which are of the form $4k+3$ and have odd exponent in $\mu$ for $K=\mathbb Q(i)$.

Then an ordered pair of vectors $(a_1,a_2)\in K^2\times K^2$ is a \emph{$Q$-orthobalanced} basis of norm $\lambda$ if
\[
A^*MA=\lambda\cdot\begin{pmatrix}
1 & 0 \\ 0 & \varepsilon
\end{pmatrix}, \qquad A:=\begin{pmatrix}
a_1 & a_2
\end{pmatrix}.
\]
Note that this notion reproduces $Q$-orthoregularity if $\varepsilon=1$, but $Q$-orthobalanced bases might exist for those hermitian forms as well, where $\det(M)$ is not an absolute square.

First we prove a rational version of Proposition \ref{perpprop}.

\begin{proposition}\label{perpprop2}~
If $(a_1,a_2)\in K^2\times K^2$ is a
$Q$-orthobalanced basis of norm $\lambda$, then $a_2=(Ma_1)_\perp/\delta$ for some $\delta\in K$, $|\delta|^2=\Delta$.
\end{proposition}
\begin{proof}
Note that for any $x,y\in K$ we have
\begin{equation}\label{Qsquare}
\alpha Q((x,y)^T)=|\alpha x+\beta y|^2+(\alpha\gamma-|\beta|^2)|y|^2=|\alpha x+\beta y|^2+\mu|y|^2.
\end{equation}
If $a_1:=(x,y)^T$ and $a_2$ form a $Q$-orthobalanced basis, 
then $a_2^*Ma_1=0$, which implies that
\begin{equation}\label{Maperpcoord}
a_2
=(Ma_1)_\perp/\delta=\left(-(\overline\beta x+\gamma y), \alpha x+\beta y\right)^*/\delta
\end{equation}
for some $\delta\in K$. Applying (\ref{Qsquare}) twice, we have that
\[
\begin{split}
\alpha Q((Ma_1)_\perp)&=\left|-\alpha\cdot(\overline{\overline\beta x+\gamma y})+\beta\cdot(\overline{\alpha x+\beta y})\right|^2+\mu\left|\overline{\alpha x+\beta y}\right|^2\\
&=\left|-\mu y\right|^2+\mu\left|\alpha x+\beta y\right|^2=\alpha\mu Q(a_1),
\end{split}
\]
hence
\[
Q((Ma_1)_\perp)=\mu Q(a_1).
\]
Finally, we obtain that
\[
\varepsilon Q(a_1)=Q(a_2)=Q((Ma_1)_\perp)/|\delta|^2=\mu Q(a_1)/|\delta|^2,
\]
which implies $|\delta|^2=\mu/\varepsilon=\Delta$, and the proof is complete.
\end{proof}

Let us call the
$Q$-orthobalanced basis $(a_1,a_2)\in K^2\times K^2$ \emph{of type $\delta$} if $a_2=(Ma_1)_\perp/\delta$.

\begin{lemma}\label{perpconverse}
Let $a_1\in K^2$, $\delta\in K$, and $a_2=(Ma_1)_\perp/\delta$ be such that $Q(a_1)=\lambda$ and $|\delta|^2=\Delta$. Then $(a_1,a_2)$ is a $Q$-orthobalanced basis of norm $\lambda$ and of type $\delta$.
\end{lemma}
\begin{proof}
If $A=\begin{pmatrix} a_1 & a_2\end{pmatrix}$ then $A^*MA$ has the form $\left(\begin{smallmatrix} \lambda & 0 \\ 0 & Q(a_2)\end{smallmatrix}\right)$ and computing determinants gives that $Q(a_2)=\lambda\varepsilon$. 
\end{proof}

Let $Q$ be the binary Hermitian form defined in the beginning of Section \ref{QuaternionSection}, $\lambda\in K_0$ and $\delta\in K$. Then we write $\mathcal B_Q(\delta,\lambda)$ for the set of all $Q$-orthobalanced bases of norm $\lambda$ and of type $\delta$. Of course, by Proposition \ref{perpprop2} and Lemma \ref{perpconverse}, this set is nonempty if and only if $\lambda\in Q(K^2)$ and $|\delta|^2 = \Delta$, but the correspondence below works even in the case when the set is empty.

The set $\mathcal{B}_Q(\delta,\lambda)$  will be related to certain factorizations in 
\begin{equation}\label{Ldef}
L:=\{r+s\sqrt{\varepsilon}j\mid r,s\in K\}\leq \mathbb H=\{a+bi+cj+dk\mid a,b,c,d\in\mathbb R\},
\end{equation}
which is isomorphic to the quadratic field $\QQ(\sqrt{-\varepsilon})$ for $K=\QQ$ and to a quaternion algebra for an imaginary quadratic $K$.

For $t\in L$ and $\nu\in \mathbb Q$ let $\mathcal F_\nu(t)$ denote the set of factorizations
\[\mathcal F_\nu(t)=\{(u,v)\in L^2\mid t=uv,|u|^2=\nu\}.\]
\begin{proposition}\label{factorizationthm}
The map $F:K^2\times K^2\to L^2$ given by
\[
F((x,y)^T,(w,z)^T):=\left(\overline z+\overline y\sqrt{\varepsilon}j,-w+\overline x\sqrt{\varepsilon}j\right)
\]
restricts to a bijecton between the sets $\mathcal B_Q(\delta,\lambda)$ and $\mathcal F_{\frac{\alpha\lambda}{\Delta}}\left(\frac{\lambda}{\Delta}(\beta+\delta\sqrt{\varepsilon}j)\right)$.
\end{proposition}

\begin{proof}
It is obvious that $F$ is a bijection from $K^2\times K^2$ to $L^2$.

If we restrict $F$ to $\mathcal B_Q(\delta,\lambda)$, then by Proposition~\ref{perpprop2},
\[
\begin{pmatrix}
w \\ z
\end{pmatrix} = \frac{1}{\delta}\left(M\begin{pmatrix}
x \\ y
\end{pmatrix}\right)_{\perp} = \frac{1}{\delta} \begin{pmatrix}
-\beta \overline{x} - \gamma \overline{y}\\ \alpha \overline{x} + \overline{\beta} \overline{y}
\end{pmatrix}.
\]

Setting
\[
u:=\overline z+\overline y\sqrt{\varepsilon}j,\qquad v:=-w+\overline x\sqrt{\varepsilon}j,
\]
we get (using (\ref{Qsquare}) in the last step),
\begin{equation}\label{u2eq}
|u|^2=|\overline z+\overline y\sqrt{\varepsilon}j|^2=\left|\frac{\alpha x+\beta y}{\overline\delta}\right|^2+\varepsilon |\overline y|^2=\frac1\Delta\left(|\alpha x+\beta y|^2+\Delta\varepsilon|y|^2\right)=\frac{\alpha\lambda}{\Delta}.
\end{equation}
On the other hand, using that for $s\in K$, we have $sj=j\overline s$, we see that
\begin{equation}\label{uveq}
\begin{split}
uv&=(\overline z+\overline y\sqrt{\varepsilon}j)(-w+\overline x\sqrt{\varepsilon}j)=-(\overline zw+\varepsilon\overline yx)+\overline{(zx-yw)}\sqrt{\varepsilon}j\\ 
&=\frac1\Delta\cdot\left(((\alpha x+\beta y)(\beta\overline x+\gamma\overline y)-\Delta\varepsilon\overline y x)+((\alpha x+\beta y)\overline x+\overline y(\overline \beta x+\gamma y))\delta\sqrt{\varepsilon}j\right) \\ 
&=\frac1\Delta(\beta Q((x,y)^T) +\delta Q((x,y)^T)\sqrt{\varepsilon}j)=\frac\lambda\Delta(\beta+\delta\sqrt{\varepsilon}j).
\end{split}
\end{equation}
Then \eqref{u2eq} and \eqref{uveq} together show that $\mathcal B_Q(\delta,\lambda)$ is indeed mapped to $\mathcal F_{\frac{\alpha\lambda}{\Delta}}\left(\frac{\lambda}{\Delta}(\beta+\delta\sqrt{\varepsilon}j)\right)$.

For the converse, assume that $(u,v)\in \mathcal{F}_{\frac{\alpha\lambda}{\Delta}}\left(\frac{\lambda}{\Delta}(\beta+\delta\sqrt{\varepsilon}j)\right)$, i.e.
\[
|u|^2=\frac{\alpha\lambda}{\Delta},\qquad uv=\frac{\lambda}{\Delta}(\beta+\delta\sqrt{\varepsilon}j),
\]
and we give explicitly $F^{-1}(u,v)\in\mathcal B_Q(\delta,\lambda)$. Writing $u=r+s\sqrt{\varepsilon} j\in L$ with $r,s\in K$, we first choose
\begin{equation}\label{eq:choice-for-u}
z:= \overline{r},\qquad y:=\overline{s},
\end{equation}
and complete this data with
\[
x:=\frac{\overline{\delta}\overline{z}-\beta y}{\alpha},\qquad w:= - \frac{\beta\overline{x}+\gamma \overline{y}}{\delta}.
\]
Applying \eqref{Qsquare}, we see that
\[\alpha Q((x,y)^T)=|(\overline\delta\overline z-\beta y)+\beta y|^2+\Delta\varepsilon|y|^2=\Delta |u|^2=\alpha\lambda,\]
hence $Q((x,y)^T)=\lambda$, and then Lemma~\ref{perpconverse} shows that $((x,y)^T,(w,z)^T)\in\mathcal B_Q(\delta,\lambda)$.

To complete the proof of this direction, we have to verify that $F((x,y)^T,(w,z)^T)=(u,v)$. As for the $u$-part, this is clear from \eqref{eq:choice-for-u} and the definition of $F$. As for the $v$-part, we have to check that $v':=-w+\overline{x}\sqrt{\varepsilon}j$ is the same as the originally given $v$. This indeed holds, since $uv=\frac{\lambda}{\Delta}(\beta+\delta\sqrt{\varepsilon}j)$ (by assumption), $uv'=\frac{\lambda}{\Delta}(\beta+\delta\sqrt{\varepsilon}j)$ (by the already proven direction), and $L$ is a division algebra.
\end{proof}

From now on assume that $M$ has entries in $R$ and $\lambda\in R$. Ultimately we are not satisfied with finding $Q$-orthobalanced bases, but we want that they consist
of integral vectors, such a setup will be referred to from now on as \emph{integral $Q$-orthobalanced} basis. Recall that $A^*MA=\lambda I$ implies that $\mu=\det M$ divides
$\lambda^2$ if $A\in R^{2\times 2}$. This, however, alone does not imply the existence of an integral $Q$-orthobalanced basis even in the simplest situation. Set for example, with $K=\QQ$,
\[
M:=\begin{pmatrix}
9 & 0 \\ 0& 1    
\end{pmatrix},
\]
then $\mu=\Delta=9$, $\varepsilon=1$, and there is no integral $Q$-orthobalanced basis of norm $3$, for then $a_1:=(x,y)^T$ should satisfy $3=9x^2+y^2$ which cannot be solved over the integers.

In the following investigations, we impose that $\Delta$ divides
$\lambda$ and it turns out that the existence of an integral $Q$-orthobalanced basis can be proven in this situation under some controllable assumptions which can be checked in the concrete situations we encounter later.

Let $S=\{r+s\sqrt{\varepsilon}j\mid r,s\in R\}\subset L$. Then $S$ is clearly an order in $L$.

\begin{proposition}\label{integralfactorizationthm}
Let $\lambda=\nu\Delta$ for some $\nu\in\mathbb Z$.
\begin{enumerate}
\item If $(a_1,a_2)\in R^2\times R^2$
is an integral $Q$-orthobalanced basis of norm $\lambda$ and $\delta$ is as in Proposition \ref{perpprop2}, then $\nu\delta\in R$.
\item There exists an integral $Q$-orthobalanced basis of norm $\lambda$ and of type $\delta\in\frac{1}{\nu} R$ (with $|\delta|^2=\Delta$) if and only if there exist $u,v\in S$ such that $uv=\nu(\beta+\delta\sqrt{\varepsilon}j)$ and $|u|^2=\alpha\nu$, $|v|^2=\gamma\nu$ (where $\alpha,\beta$ and $\gamma$ are the entries of $M$).
\item If $K$ is of class number one, then any $a_1\in R^2$
with $Q(a_1)=\lambda$ can be extended to an integral $Q$-orthobalanced basis of norm $\lambda$.
\end{enumerate}
\end{proposition}

\begin{proof} We will use the notations and the results of Proposition~\ref{factorizationthm}.
\begin{enumerate}
\item Assume that $(a_1,a_2)\in R^2\times R^2$ is an integral $Q$-orthobalanced basis of norm $\lambda$ with $\delta$ as in Proposition \ref{perpprop2}. Setting then $(u,v):=F(a_1,a_2)$, we see that $u,v\in S$,
hence $uv=\nu(\beta+\delta\sqrt{\varepsilon}j)\in S$, which implies $\nu\delta\in R$.
\item Assume there exist $u,v\in S$ such that $uv=\nu(\beta+\delta\sqrt{\varepsilon}j)$ and $|u|^2=\alpha\nu$, $|v|^2=\gamma\nu$. Considering $(a_1,a_2):=F^{-1}(u,v)$, from the explicit definition of $F$ we see that $F^{-1}(S^2)\subseteq R^2\times R^2$, hence the pair $(a_1,a_2)$ is an integral $Q$-orthobalanced basis of norm $\lambda$. To see that $\delta\in\frac{1}{\nu}R$, we refer to the already proven part (1). As for the converse, assume that $(a_1,a_2)\in R^2\times R^2$ is an integral $Q$-orthobalanced basis of norm $\lambda$ and of type $\delta\in\frac{1}{\nu} R$ (with $|\delta|^2=\Delta$). Then for $(u,v):=F(a_1,a_2)$, we see immediately $uv=\nu(\beta+\delta\sqrt{\varepsilon}j)$ and $|u|^2=\alpha\nu$ from Proposition~\ref{factorizationthm}, while
\[
|v|^2=\frac{|uv|^2}{|u|^2}=\frac{\nu^2(\beta+\delta\sqrt{\varepsilon}j)(\overline{\beta}-\overline{\delta}\overline{\sqrt{\varepsilon}}j)}{\alpha\nu}=\frac{\nu^2(|\beta|^2+|\delta|^2\varepsilon)}{\alpha\nu}=\frac{\nu^2\alpha\gamma}{\alpha\nu}=\gamma\nu.
\]
\item Let $a_1=(x,y)^T$ and $\nu(Ma_1)_\perp=(w_0,z_0)^T$.
Substituting $z_0:=\nu\delta'z$, $w_0:=\nu\delta'w$
with some $\delta'\in K$ satisfying $|\delta'|^2=\Delta$, we see that $(a_1,(w,z)^T)$ is a $Q$-orthobalanced basis of norm $\lambda$ and type $\delta'$ by Lemma \ref{perpconverse}. Hence we may use \eqref{u2eq} and \eqref{uveq} to compute
\begin{equation*}\label{eq:z0w0}
\begin{split}
&|z_0|^2=\nu^2\Delta(\alpha\nu-\varepsilon|y|^2) 
,\\
&|w_0|^2=\nu^2\Delta(\gamma\nu-\varepsilon|x|^2) 
,\\
&-\overline z_0\cdot w_0=\nu^2\Delta(\beta\nu+\varepsilon\overline yx)
.
\end{split}
\end{equation*}
In fact, the first line here is implicit in \eqref{u2eq}, the second line is a simple analogue, while the third line is implicit in \eqref{uveq}. We construct the second column vector in the form $a_2:=(w_0/\delta_0,z_0/\delta_0)^T$ where $\delta_0\in \mathcal{O}_K$ divides both $z_0$ and $w_0$ and $|\delta_0|^2=\nu^2\Delta$.

Since $K$ has class number one, $\mathcal{O}_K$ has unique factorization. For each rational integral prime divisor $p$ of $\nu^2\Delta$ we compare the exponents of the primes above $p$ in both sides. Assume $\nu_p(\nu^2\Delta)=r$, where $\nu_p$ stands for the $p$-adic valuation.
\begin{enumerate}
\item If there is a single prime $\pi\in R$ which divides $p$, then $p=\pi^s$ with $s=1$ or $2$. Note that if $s=1$ then $2\mid r$. Set $\delta_p=\pi^{rs/2}$. As $|z_0|^2$ and $|w_0|^2$ are both divisible by $p^r$, we have $\delta_p\mid z_0$ and $\delta_p\mid w_0$.
\item If $p$ splits, 
write $p=\pi\overline\pi$ such that $\overline{\pi}$ is not a unit multiple of $\pi$.
Write $\nu_{\pi}(w_0):=\rho_w$, $\nu_{\overline{\pi}}(w_0):=\sigma_w$, $\nu_{\pi}(z_0):=\rho_z$ and $\nu_{\overline{\pi}}(z_0):=\sigma_z$, where $\nu_{\pi}$ and $\nu_{\overline{\pi}}$ stand for the $\pi$-adic and $\overline{\pi}$-adic valuation, respectively. Then by the above equations we have $\rho_w+\sigma_w\geq r$, $\rho_z+\sigma_z\geq r$, $\rho_w+\sigma_z\geq r$ and $\rho_z+\sigma_w\geq r$. Set $\rho=\min(\rho_w,\rho_z)$, $\sigma=r-\rho$ and $\delta_p=\pi^\rho\overline\pi^\sigma$. Then $\delta_p\mid z_0$ and $\delta_p\mid w_0$.
\end{enumerate}
Setting then
\[\delta_0:=\prod_{p|\nu^2\Delta}\delta_p,\]
we see that $\delta_0\in \mathcal{O}_K$ and $|\delta_0|^2=\nu^2 \Delta$. Writing $\delta:=\delta_0/\nu$, this implies that $\delta\in\frac{1}{\nu} R$ and $|\delta|^2=\Delta$. The proof is then complete by Lemma~\ref{perpconverse}.
\end{enumerate}
The proof is complete.
\end{proof}

\begin{proof}[Proof of Theorem \ref{OrthoregbasisthmK}]
The statement of the theorem is the special case $\varepsilon=1$ of (3) of Proposition \ref{integralfactorizationthm}.
\end{proof}

Motivated by Proposition~\ref{integralfactorizationthm}(2), we introduce the following.
\begin{definition}
We say that the order $S$ defined above \emph{has enough divisors} if
the following holds. Whenever $|y|^2$ divides $|t|^2$ for some $y,t\in S$, then there exists a (left) divisor $u\in S$ of $t$ such that $|u|^2=|y|^2$. This property depends on both $K$ and $\varepsilon$.
\end{definition}

\begin{corollary}\label{enoughdivcorollary}
Assume that $S$ has enough divisors. Then for any hermitian form $Q$ of discriminant $\mu=\Delta\varepsilon$ (as in (\ref{normfree})) with matrix $\left(\begin{smallmatrix}\alpha & \beta \\ \overline\beta & \gamma\end{smallmatrix}\right)\in R^{2\times 2}$ and for any $\lambda=\nu\Delta$ with $\nu\in\ZZ$ and $\delta\in\frac{1}{\nu} R$ with $|\delta|^2=\Delta$, the following are equivalent:
\begin{enumerate}
\item There exists an integral $Q$-orthobalanced basis of norm $\lambda$ and type $\delta$.
\item $\alpha\nu
=|y|^2$ for some $y\in S$.
\end{enumerate}
\end{corollary}

\begin{proof}
The direction $(1)\Rightarrow(2)$ immediately follows from Proposition~\ref{integralfactorizationthm}(2) even without the assumption that $S$ has enough divisors. As for the converse $(2)\Rightarrow(1)$, consider $\alpha\nu,\nu^2\alpha\gamma\in\ZZ$, which satisfy
\[
\alpha\nu=|y|^2,\qquad \nu^2\alpha\gamma=\nu(\beta+\delta\sqrt{\varepsilon}j)\overline{\nu}(\overline{\beta}-\overline{\delta}\overline{\sqrt{\varepsilon}}j)=|t|^2,
\]
for some $y,t\in S$. Since $S$ has enough divisors, we can factorize $\nu(\beta+\delta\sqrt{\varepsilon}j)$ to $uv$ satisfying $|u|^2=\alpha\nu$, $|v|^2=\gamma\nu$, and the respective implication of Proposition~\ref{integralfactorizationthm}~(2) applies.
\end{proof}

\begin{lemma}\label{enoughdivlemma}
If $R=\mathbb Z$ or $\mathbb Z[i]$ and $\varepsilon=1$ then 
\[S=\begin{cases}\mathbb Z[j], & \text{ if } R=\mathbb Z\\ \mathbb L=\{r+sj\mid r,s\in\mathbb Z[i]\}, & \text{ if }R=\mathbb Z[i].\end{cases}\]
In both cases $S$ has enough divisors.
\end{lemma}

\begin{proof}
It is easy to verify that $S$ is as in the statement. We are left with checking that the condition in the definition of ``$S$ has enough divisors'' holds. Assume hence that $|y|^2$ 
divides $|t|^2$ for some $y,t\in S$, our goal is to factorize $t$ to $uv$ such that $|u|^2=|y|^2$.

The case $S=\ZZ[j]$ follows from Lemma~\ref{quotientisnorm}, hence we focus on $S=\mathbb{L}$. Since the case $|y|^2=1$ is trivial, a simple induction argument shows that it suffices to the following situation. Let $p$ is a rational prime $t\in S$ such that $p\mid |t|^2$. Then there exist $u,v\in S$ such that $|u|^2=p$ and $t=uv$. We go by cases.

Case 1: $p=2$. Then for $t=t_{\RR}+t_i i+t_j j+ t_k k$ (with $t_{\RR},t_i,t_j,t_k\in\ZZ$), we have that $2\mid |t|^2=t_{\RR}^2+t_i^2+t_j^2+t_k^2$, hence $t_{\RR}$ has the same parity as odd many among $t_i,t_j,t_k$. If, say, $t_{\RR}\equiv t_i\bmod 2$, then $t_j\equiv t_k \bmod 2$, and one can easily check that
\[(1+i)^{-1}t=\frac{t_{\RR}+t_i}{2}+\frac{-t_{\RR}+t_i}{2}i+\frac{t_j+t_k}{2}j+\frac{-t_j+t_k}{2}k\in S,\]
hence $u=1+i$ and $v=(1+i)^{-1}t$ give an appropriate factorization admitting $|u|^2=2$. If $t_{\RR}\equiv t_j \bmod 2$, then similarly $u=1+j$, $v=(1+j)^{-1}t$, while if $t_{\RR}\equiv t_k \bmod 2$, then $u=1+k$, $v=(1+k)^{-1}t$ are appropriate.

Case 2: $p\neq 2$ and $|t|^2<p^2$. Let $S':=S\cup(S+\frac{1+i+j+k}{2})$ be the order of Hurwitz quaternions. It is known that every nonzero right ideal of $S'$ is a principal right ideal (in fact, there is a ``non-commutative euclidean division'', see e.g. \cite[Section~5.1]{ConwaySmith}). This implies that, for some $u'\in S'$,
\[
pS'+tS'=u'S'.
\]
Then $p=u'c$ for some $c\in S'$, hence $|u'|^2\mid p^2$ (in $\ZZ$). Similarly, $t=u'v'$ for some $v'\in S'$, hence $|u'|^2\mid |t|^2$ (in $\ZZ$). Recalling we are under the assumption $|t|^2<p^2$, these altogether imply that $|u'|^2\in\{1,p\}$. On the other hand, we have that
$u'=pa+tb$ with some $a,b\in S'$, hence
\[
|u'|^2=p^2|a|^2+p(a\overline{tb}+tb\overline{a})+|t|^2|b|^2.
\]
By assumption, $p\mid |t|^2$, hence $|u'|^2=p$.

If $u'\in S$, then we set $u:=u',v:=v'$.

If $u'\in S'\setminus S$, then choose $\omega=\frac{\pm 1 \pm i \pm j \pm k}{2}$ in such a way that $(u'+\omega)/2\in S$ (i.e.~$u'+\omega$ has even integer coordinates in the basis $1,i,j,k$). Let then
\[
u:=u'\overline{\omega}=(u'+\omega)\overline{\omega}-1=\frac{u'+\omega}{2}\cdot 2\overline{\omega}-1.
\]
Then on the one hand, $|u|^2=|u'|^2=p$, while on the other hand, $u\in S$. Correspondingly, let $v:=\omega v'$.

In either case, $uv=u'v'=t$ is clear. Further, $u$ has odd many odd coordinates (in the standard basis $1,i,j,k$), since $|u|^2=p$ is odd. This implies that $v\notin S'\setminus S$ (for if $v\in S'\setminus S$, then $t=uv\in S'\setminus S$, a contradiction). Then $uv=t$, $u,v\in S$ and $|u|^2=p$.

Case 3: $p\neq 2$ and $|t|^2\geq p$. Choose $m\in S$ such that
\[
t':=t-pm\in \left(-\frac{p}{2},\frac{p}{2}\right) + \left(-\frac{p}{2},\frac{p}{2}\right)i+ \left(-\frac{p}{2},\frac{p}{2}\right)j+\left(-\frac{p}{2},\frac{p}{2}\right)k,
\]
in particular, $|t'|^2<p^2$. Then \[|t'|^2=|t-pm|^2=(t-pm)(\overline t-p\overline m)\equiv |t|^2\equiv 0 \bmod p.\] We know from Case 2 that $t'=uv$ with some $u,v\in S$, $|u|^2=u\overline u=p$. Then $t=uv+pm=u(v+\overline{u}m)$ is an appropriate factorization, since $v+\overline{u}m\in S$.
\end{proof}

\begin{remark}
In \cite[Section 5]{ConwaySmith} it is proven that for any (Hurwitz or Lipschitz) integral quaternion $t$ and any factorizaton $|t|^2=p_1p_2\dots p_k\in\mathbb N$, where the $p_j$ are rational primes, there exists a factorization $t=P_1P_2\dots P_k$ modeled on the rational factorization, that is $|P_j|^2=p_j$ for all $j$. Also the number of such factorizations is determined. However, to make the paper self contained we decided to include a proof, which does not rely on this theory.
\end{remark}

\begin{proof}[Proof of Proposition~\ref{perpprop}]
Assume $(a_1,a_2)\in R^2\times R^2$ is a $Q$-orthoregular basis of norm $\lambda=\nu\Delta$ as in the statement of Proposition~\ref{perpprop}. Then applying Proposition~\ref{perpprop2}, we see that $a_2=(Ma_1)_{\perp}/\delta$ with some $\delta\in K$, $|\delta|^2=\Delta$. Also, $\delta\in\frac{1}{\nu} R$ by Proposition~\ref{integralfactorizationthm}(1).
\end{proof}

\begin{proof}[Proof of Proposition~\ref{Orthoregbasisthm}]
Assume $Q,\lambda,\nu,\Delta,\delta$ are as in the statement of Proposition~\ref{Orthoregbasisthm}. Applying Lemma~\ref{enoughdivlemma}, we see that the corresponding orders $\ZZ[j]$ or $\mathbb{L}$ have enough divisors, and then Corollary~\ref{enoughdivcorollary} yields Proposition~\ref{Orthoregbasisthm}(1)-(2). The statement Proposition~\ref{Orthoregbasisthm}(3) is immediate from Proposition~\ref{integralfactorizationthm}(3) by recalling that the class number of both $\QQ$ and $\QQ(i)$ are $1$.
\end{proof}

\begin{example} We exhibit some examples to justify some of the above notions and conditions.
\begin{enumerate}
\item $S$ does not have enough divisors in general, for example if $K=\mathbb Q$ and $\varepsilon=17$, then $S=\mathbb Z[\sqrt{17}j]$, $\alpha=21$ is an absolute square (for example $|2+\sqrt{17}j|^2=21$) and $t=4\pm5\sqrt{17}j$ is of absolute square $21^2$, but has no divisor of absolute square $21$.
\item The fact that $\alpha\nu$ is an absolute square is not sufficient for the existence of an integral $Q$-orthobalanced basis of type $\delta$ and norm $\lambda$ in general. For example set $K=\mathbb Q$ and consider the quadratic form $Q$ corresponding to the matrix $M=\left(\begin{smallmatrix}21 & 4 \\ 4 & 21\end{smallmatrix}\right)$ with discriminant $\mu=425$. Then $\Delta=25$ and $\varepsilon=17$ and there is no integral $Q$-orthobalanced basis of length $\lambda=25$ (with $\nu=1$), as the only possible values for $\delta$ are $\pm5$, but $\beta+\delta j=4\pm5\sqrt{17}j$ have no divisor of $u$ with $|u|^2=21$.
\item One cannot extend an integral vector to an integral $Q$-orthobalanced basis in general. Consider the hermitian form $Q$ corresponding to the matrix $\left(\begin{smallmatrix}5 & 2 \\ 2 & 5\end{smallmatrix}\right)$ over $K=\mathbb Q(\sqrt{-17})$. Then as $\mu=21$ is an absolute square in $R=\mathcal O_K$, we have $\Delta=21$, $\varepsilon=1$. Let $a_1:=(\sqrt{17}i,2)^T\in R^2$ with $\lambda=Q(a_1)=105$, thus $\nu=5$. With the notation used in the proof of Proposition \ref{integralfactorizationthm}(3) we have $z_0=5(4-5\sqrt{17}i)$, which -- as an easy computation shows -- has no divisor $y$ with $|y|^2=\nu^2\Delta=25\cdot 21$, hence $a_1$ cannot be extended to an integral $Q$-orthobalanced basis.
\end{enumerate}
\end{example}

Later we will need the following

\begin{lemma}\label{squaredetlem}
Let $R=\mathbb Z$ and $Q$ be a positive definite hermitian form with matrix
\[
M:=\begin{pmatrix} \alpha & \beta \\ \beta & \gamma \end{pmatrix} \in\ZZ^{2\times 2}
\]
as before, and assume that $\det(M)=\Delta$ is a square. Then the following are equivalent:
\begin{enumerate}
\item $\alpha$ is the sum of two squares,
\item $Q(x)$ is the sum of two squares for all $x\in\mathbb Z^2$,
\item $Q(x)$ is the sum of two squares for some 
$x\in\mathbb Z^2\setminus\{(0,0)^T\}$.
\end{enumerate}
Otherwise, there exists a prime $q=4k+3\in\mathbb Z$ such that for all $x\in\mathbb Z^2\setminus\{(0,0)^T\}$ the exponent of  $q$ in $Q(x)$ is odd.
\end{lemma}

\begin{proof}
As $M$ is positive definite, we have $\alpha>0$. By (\ref{Qsquare}) and the fact that $\mu=\Delta$ is a square, we have that $\alpha Q(x)$ is a sum of two squares for any $x\in\mathbb Z^2$. Also Lemma~\ref{quotientisnorm} applied to $K:=\QQ(i)$ gives that the quotient, provided it is an integer, of two rational integers, which individually are the sum of two squares, is again the the sum of two squares. Recalling these, all implications are simple.

$(1) \Rightarrow (2)$: we have that $Q(x)=\alpha Q(x)/\alpha,$
the right-hand side is the sum of two squares, so is the left-hand side.

$(2) \Rightarrow (3)$: obvious.

$(3) \Rightarrow (1)$: since $Q$ is positive definite, $Q(x)\neq 0$ for the specified $x\neq (0,0)^T$ for which $Q(x)$ is the sum of two squares. Then $\alpha=\alpha Q(x)/Q(x),$ the right-hand side is the sum of two squares, so is the left-hand side.

If $\alpha$ is not a sum of two squares, then it has a prime divisor $q=4k+3$ which has odd exponent in $\alpha$. Then again by the fact that $\alpha Q(x)$ is a sum of two squares we have that the exponent of $q$ in $Q(x)$ is odd for all $x\in\mathbb Z^2\setminus\{(0,0)^T\}$.
\end{proof}

\section{Extension of icubes}\label{SecGenExt}

From now on let $K=\mathbb Q$ or $\mathbb Q(i)$ and $R=\mathbb Z$ or $\mathbb Z[i]$ again.

Let $A_0=(a_1|a_2|\dots|a_k)\in R^{n\times k}$ be an icube of norm $\lambda$. We will need the following:

\begin{proposition}\label{discQlemma}~
\begin{enumerate}
\item The $R$-module
\[\Lambda=\{w\in R^n \mid a_j^*w=0\text{ for }1\leq j\leq k\}\]
is free of rank $n-k$.
\item If $Q$ is the restriction of the standard inner product to $\Lambda$, then
\[\disc(Q)=\frac{\lambda^k}{|d_k(A_0)|^2},\]
where $\disc(Q)$ is the determinant of the Gram matrix of $Q$ with respect to a free basis as in (1) and $d_k(A_0)$ is the $k$-th determinantal divisor of $A_0$.
\end{enumerate}
\end{proposition}

The case $R=\mathbb Z$, $k=1$ and $a_1\in R^n$ is primitive ($d_1(A_0)=1$) is \cite[Proposition~2]{Horvath}.

\begin{proof}~
\begin{enumerate}
\item As $R$ is a principal ideal domain, $\Lambda$ is free. On the other hand, we have
\[\dim(\Lambda\otimes K)=\dim(\mathrm{Span}(a_1,a_2,\dots,a_k)^\perp)=n-k,\]
so $\Lambda$ has rank $n-k$.
\item Fix a basis $b_{k+1},b_{k+2},\dots,b_n$ of $\Lambda$ and let
\[\Lambda'=\langle a_1,a_2,\dots,a_k,b_{k+1},\dots,b_n\rangle\leq R^n,\]
a free module of rank $n$ and put $B:=(a_1|a_2|\dots|a_k|b_{k+1}|\dots |b_n)$. Then $B^*B=\left(\begin{smallmatrix} \lambda I & \\ & M\end{smallmatrix}\right)$, where $I\in R^{k \times k}$ is the identity matrix and $M=M^*\in R^{(n-k)\times (n-k)}$ is the matrix of $Q$ in the basis $b_{k+1},b_{k+2},\dots, b_n$. Thus
\begin{equation}\label{LambdaDisc}
\disc(Q)=\det(M)=|\det(B)|^2/\lambda^k.
\end{equation}

The $R$-index $[R^n:\Lambda']_R$ is defined (\cite{Conrad}, Definition 5.14) as the $R$-cardinality (or characteristic ideal) of the quotient $R^n/\Lambda'$ and equals (\cite{Conrad}, Theorem 5.22)
\[[R^n:\Lambda']_R=(\det(B)).
\]

Consider the map
\[\varphi:R^n\to R^k, \quad w\mapsto A_0^*w=(a_1^*w,a_2^*w,\dots,a_k^*w)^T.\]
Its kernel is $\Lambda$ by definition and its image $\varphi(R^n)\leq R^k$ is the $R$-module generated by the images $\varphi(e_i)=A_0^*e_i$ of the standard basis vectors for $1\leq i\leq n$. As $\varphi(a_j)=\lambda e_j$ for $1\leq j\leq k$ we have
\[\varphi(\Lambda')=(\lambda R)^k\leq\varphi(R^n).\]
Thus $\varphi$ restricts to an isomorphism $R^n/\Lambda'\to\varphi(R^n)/(\lambda R)^k$. Using this and \cite{Conrad}, Theorem 5.18 we get
\begin{equation}\label{Rindex}
(\lambda^k)=[R^k:(\lambda R)^k]_R=[R^k:\varphi(R^n)]_R[(\varphi(R^n):(\lambda R)^k]_R=[R^k:\varphi(R^n)]_R[R^n:\Lambda']_R.
\end{equation}

We claim $[R^k:\varphi(R^n)]_R=(\overline{d_k(A_0)})$. Indeed if $b'_1,\dots, b'_k\in R^k$ are free generators of $\varphi(R^n)$ then they are $R$-linear combinations of the images $\varphi(e_i)=A_0^*e_i$ of the standard basis vectors $1\leq i\leq n$. Using \cite{Conrad}, Theorem 5.22 again we get $[R^k:\varphi(R^n)]_R=\det(b'_1|\dots|b'_k)$ and by the multilinearity of the determinant it is an $R$-linear combination of $k\times k$ minors of $A_0^*$. Hence $(d_k(A_0^*))=(\overline{d_k(A_0)})$ divides 
$[R^k:\varphi(R^n)]_R$.

On the other hand any column $A_0^*e_i$ of $A_0^*$ can be expressed as an $R$-linear combination of the $b'_j$-s, so the determinant of any $k\times k$ minor is divisible by $\det(b'_1|b'_2|\dots|b'_k)$, thus
$[R^k:\varphi(R^n)]_R=(\det(b'_1|b'_2|\dots|b'_k))$ also divides $(d_k(A_0^*))=(\overline{d_k(A_0)})$ and our claim follows.

Plugging this into \eqref{Rindex} we obtain $[R^n:\Lambda']_R=(\lambda^k/\overline{d_k(A_0)})$ whence
\begin{equation}\label{detBuptounit}
\det(B)=\omega\lambda^k/\overline{d_k(A_0)} 
\end{equation}
for some unit $\omega\in R$. Using \eqref{LambdaDisc} we finally deduce
\begin{equation}\label{discQform}
\disc(Q)=\frac{|\omega\lambda^k/\overline{d_k(A_0)}|^2}{\lambda^k}=\frac{\lambda^k}{|d_k(A_0)|^2}.
\end{equation}
\end{enumerate}
\end{proof}

\begin{corollary}
We have that $d_k(A_0)\mid \lambda^k$ (in $R$).
\end{corollary}

With the notation of the previous proof we also have

\begin{proposition}\label{crossproductprop}
Assume $A_0=(a_1|a_2|\dots|a_k)$ is a $k$-icube of norm $\lambda$ with $a_1$ primitive. Fix a basis $(b_{k+1},b_{k+2},\dots,b_n)$ of $\Lambda$ and set $B=(a_1|a_2|\dots|a_k|b_{k+1}|\dots|b_n)$. Let $a=(\alpha_1,\alpha_2,\dots,\alpha_n)^T$ be the cross product of $a_2,a_3,\dots,a_k,b_{k+1},\dots,b_n$, that is
\[\alpha_j=\overline{\det (B_{j,1})} ,\]
where $B_{j,1}$ is the cofactor we get by deleting the $j$th row and first column of $B$. Then $a=(\overline{\det(B)}/\lambda)\cdot a_1$ with $\overline{\det(B)}/\lambda\in R$.
\end{proposition}

\begin{proof}
Note that the first row of $B^{-1}$ is $\frac{1}{\det (B)}a^*$. In particular, both $a$ and $a_1$ are orthogonal to each of the linearly independent vectors $a_2,a_3,\dots,a_k, b_{k+1},\dots,b_n$ by construction,
so we may write $a=\rho a_1$ for some $\rho\in K$. On the other hand, computing the top left entry of $B^{-1}B$ in two ways, we obtain
\begin{equation}\label{computerho}
1=\frac{1}{\det(B)}a^*a_1=\frac{\overline{\rho}}{\det(B)}a_1^*a_1=\frac{\overline{\rho}\lambda}{\det(B)},
\end{equation}
hence $\rho=\overline{\det(B)}/\lambda$. Finally note that $\rho\in R$ as $a_1$ is primitive and $a\in R^n$.
\end{proof}

\subsection{Dimension 3}\label{Dim3SSection}~
\begin{proof}[Proof of Theorems~\ref{3Ziv},~\ref{3Z}~and~\ref{3Zi}]
Let $A_0$ be a $k$-icube (with $1\leq k\leq 3$ in $R^3$) of norm $\lambda$. Our goal is to extend it to a $3$-icube.

Case 1: $k=1$, i.e.~$A_0=(a_1)$. First we assume that 
$a_1$ is primitive. 
Consider then $\Lambda:=\{w\in R^3\mid a_1^*w=0\}$, a free module of rank $2$ by Proposition~\ref{discQlemma}(1). Consequently $\Lambda$ will sometimes be implicitly identified with $R^2$, especially when applying Proposition~\ref{Orthoregbasisthm}.

Denoting by $e_1,e_2,e_3$ the standard generators of $R^3$ we note that $\Lambda$ intersects $Re_1+Re_2$ nontrivially. In particular, there is an element $0\neq b_2\in \Lambda$ whose $3$rd coordinate is zero. Dividing by the greatest common divisor of the entries of $b_2$ we may assume $b_2$ is primitive whence it can be extended by a vector $b_3$ to a basis of $\Lambda$.

Let $Q$ be the restriction of the standard inner product to $\Lambda$. By Proposition~\ref{discQlemma}(2), $\disc(Q)=\lambda$.

If $R=\ZZ[i]$, take some $\delta\in R$ such that $\lambda=|\delta|^2$, and apply Proposition~\ref{Orthoregbasisthm}(2) with $Q,\lambda,\nu:=1,\Delta:=\lambda,\delta$ to get an integral orthoregular basis $(a_2,a_3)$ of norm $\lambda$ in $\Lambda$. Then $(a_1|a_2|a_3)$ is a $3$-icube of norm $\lambda$, an extension of $A_0$.

If $R=\ZZ$, then write $M:=\left(\begin{smallmatrix}
\alpha & \beta \\ \beta & \gamma
\end{smallmatrix}\right)$ for the Gram matrix of $Q$ in the basis $b_2,b_3$. Then $\alpha=Q(b_2)=b_2^*b_2$ is the sum of two squares, since $b_2$ has a zero coordinate. Then we apply Proposition~\ref{Orthoregbasisthm}(1) with $Q,\lambda,\nu:=1,\Delta:=\lambda,\delta$ to get an integral orthoregular basis $(a_2,a_3)$ of norm $\lambda$ in $\Lambda$. Then $(a_1|a_2|a_3)$ is a $3$-icube of norm $\lambda$, an extension of $A_0$.

If the entries of $a_1$ are not jointly coprime, say, their $\gcd$ is $\mu$, then the $1$-icube $a_1/\mu$ can be extended to a $3$-icube $((a_1/\mu)|a_2|a_3)$ of norm $\lambda/|\mu|^2$, as already proven. Therefore, $(a_1|(\mu a_2)|(\mu a_3))$ is an extension of $A_0$ to a $3$-icube of norm $\lambda$.

Case 2: $k=2$, i.e.~$A_0=(a_1|a_2)$. Consider again $\Lambda:=\{w\in R^3\mid a_1^*w=0\}$, and let $Q$ be the restriction of the standard inner product to $\Lambda$. Then $a_2\in \Lambda$ satisfies $Q(a_2)=\lambda$, hence by Proposition~\ref{Orthoregbasisthm}(3), it can be extended to a $Q$-orthoregular basis $(a_2,a_3)$ in $\Lambda$. Then $(a_1|a_2|a_3)$ is a $3$-icube extension of $A_0$.
\end{proof}

\subsection{Dimension 4} \label{Dim4SSection}
We start this section be proving Theorems~\ref{4Z}~and~\ref{4Zi}. We will rely on Propositions~\ref{d2yx}~and~\ref{sumsquared4}, which will be treated afterwards.
\begin{proof}[Proof of Theorems~\ref{4Ziv},~\ref{4Z}~and~\ref{4Zi}]
We prove that any $k$-icube $A_0$ of norm $\lambda$ can be extended to a $(k+1)$-icube for $1\leq k\leq 3$.

If $k=1$ then let $A_0=(\alpha_1,\alpha_2,\alpha_3,\alpha_4)^T\in R^{4\times 1}$. Then the following is an icube:
\begin{equation}\label{{eq:1-icube-to-2-icube-in-Z[i]4}}
\begin{pmatrix}\alpha_1 & \alpha_2 & \alpha_3 & \alpha_4\\ -\overline{\alpha_2} & \overline{\alpha_1} & -\overline{\alpha_4} & \overline{\alpha_3}\end{pmatrix}^T
\end{equation}

If $k=2$ then consider the $R$-module $\Lambda$ and the binary hermitian form $Q$ as in Proposition~\ref{discQlemma}. By Proposition~\ref{discQlemma}(2) we have $\disc(Q)=\lambda^2/|d_2(A_0)|^2$. 

\begin{proposition}\label{d2yx}
For any $x,y\in\Lambda$ we have that $|d_2(A_0)|^2y^*x$ is divisible by $\lambda$.
\end{proposition}

Using this we get that the quadratic form $Q':=\frac{|d_2(A_0)|^2}\lambda Q$ has an integral matrix. We have
\begin{equation}\label{discQprime}
\disc(Q')=\left(\frac{|d_2(A_0)|^2}{\lambda}\right)^2\disc(Q)=|d_2(A_0)|^2.
\end{equation}

\begin{proposition}\label{sumsquared4}
Assume $A_0=(a_1|a_2)\in R^{4\times 2}$ is a $2$-icube. Then there exists an integral $Q'$-orthoregular basis of norm $|d_2(A_0)|^2$. 
\end{proposition}

However, the basis provided by Proposition \ref{sumsquared4} is an integral $Q$-orthoregular basis of norm $\lambda$, so any element of it extends $a_1,a_2$ to a 3-icube, finishing the $k=2$ case.

For $k=3$ we do the same as above for the first two vectors $a_1, a_2$. Then $a_3$ lies in $\Lambda$ and satisfies $Q'(a_3)=|d_2(A_0)|^2$, so by Proposition \ref{Orthoregbasisthm}(3) it can be extended to a $Q'$-orthoregular integral basis with a vector $a_4$, for which $A=(a_1|a_2|a_3|a_4)$ is the required 4-icube.
\end{proof}

For the proof of Propositions \ref{d2yx} and \ref{sumsquared4} we need the following notation: Let $A_0:=(a_1|a_2)\in R^{4\times 2}$ be a 2-icube of norm $\lambda$ with $a_1=(\alpha_1,\alpha_2,\alpha_3,\alpha_4)^T$, 
$a_2=(\beta_1,\beta_2,\beta_3,\beta_4)^T$, then set  
$d_{e,f}:=\alpha_e\beta_f-\beta_e\alpha_f$ for $1\leq e,f\leq 4$. For an integer $h\in\{1,2,3,4\}$ let us denote by $x_h$ the  cross product (see Proposition \ref{crossproductprop}) of $\pi_h(a_1)$ and $\pi_h(a_2)$ where $\pi_h$ is the orthogonal projection to the orthogonal complement of the $h$th basis vector. For example $x_4=(\overline{d_{2,3}},\overline{d_{3,1}},\overline{d_{1,2}},0)^T$. 
\begin{lemma}\label{generatelemma}
We have \[\overline{d_2(A_0)}\Lambda=Rx_1+Rx_2+Rx_3+Rx_4.\]
\end{lemma}
\begin{proof}
By construction $x_1,x_2,x_3,x_4\in \Lambda$, as each of these vectors are orthogonal to both $a_1$ and $a_2$. Further, all entries of $x_1,x_2,x_3,x_4$ are divisible by $\overline{d_2(A_0)}$, so we have $x_1,x_2,x_3,x_4\in \overline{d_2(A_0)}\Lambda$.

On the other hand, pick $y=(\gamma_1,\gamma_2,\gamma_3,\gamma_4)^T\in \Lambda$. We claim that $\overline{d_{e,f}}y\in Rx_1+Rx_2+Rx_3+Rx_4$ for all $1\leq e<f\leq 4$. By symmetry, we may assume without loss of generality that $e=1,f=2$. If $d_{1,2}=0$ then there is nothing to prove, so assume $d_{1,2}\neq 0$ so that $(\alpha_1,\alpha_2,0,0)$ is \emph{not} parallel to $(\beta_1,\beta_2,0,0)$. However, the $3$rd and $4$th entries of $\overline{d_{1,2}}y-\gamma_3x_4-\gamma_4x_3$ are both $0$ while this vector is orthogonal to both $a_1$ and $a_2$ hence also orthogonal to both $(\alpha_1,\alpha_2,0,0)$ and $(\beta_1,\beta_2,0,0)$. We deduce $\overline{d_{1,2}}y-\gamma_3x_4-\gamma_4x_3=0$ whence $\overline{d_{1,2}}y\in Rx_1+Rx_2+Rx_3+Rx_4$ as claimed. 

Finally, note that the greatest common divisor of $\overline{d_{e,f}}$ ($1\leq e<f\leq 4$) is $\overline{d_2(A_0)}$, so we obtain $\overline{d_2(A_0)}y\in Rx_1+Rx_2+Rx_3+Rx_4$.
\end{proof}

\begin{lemma}\label{defLemma}
For all $1\leq g,h\leq 4$ we have
\[x_h^*x_g=\begin{cases}\lambda(\lambda-|\alpha_h|^2-|\beta_h|^2)&\text{if }g=h\\
\lambda(\alpha_g\overline{\alpha_h}+\beta_g\overline{\beta_h}) &\text{if }g\neq h .
\end{cases}\]
In particular, $\lambda$ divides $x_g^*x_h$.
\end{lemma}

\begin{proof}~
At first we treat the case $g=h$. By Lagrange's identity
\begin{align*}
x_h^*x_h 
&=|\pi_h(a_1)\times \pi_h(a_2)|^2=|\pi_h(a_1)|^2|\pi_h(a_2)^2|-|\pi_h(a_2)^*\pi_h(a_1)|^2\\ &=
(\lambda-|\alpha_h|^2)(\lambda-|\beta_h|^2)-|\alpha_h|^2\cdot|\beta_h|^2\\ &=\lambda(\lambda-|\alpha_h|^2-|\beta_h|^2).
\end{align*} 
Now assume $g\neq h$ and let $e$ and $f$ be the two remaining elements in $\{1,2,3,4\}$ so that $e,f,g,h$ is a permutation of the set $\{1,2,3,4\}$. By definition
\begin{align*}
x_h^*x_g=d_{e,g}\overline{d_{e,h}}+d_{f,g}\overline{d_{f,h}}&=
(\alpha_e\beta_g-\beta_e\alpha_g)(\overline{\alpha_e}\overline{\beta_h}-\overline{\beta_e}\overline{\alpha_h})+(\alpha_f\beta_g-\beta_f\alpha_g)(\overline{\alpha_f}\overline{\beta_h}-\overline{\beta_f}\overline{\alpha_h})\\ &=
(|\alpha_e|^2+|\alpha_f|^2)\beta_g\overline{\beta_h}+(|\beta_e|^2+|\beta_f|^2)\alpha_g\overline{\alpha_h}\\ & \qquad -(\alpha_e\overline{\beta_e}+\alpha_f\overline{\beta_f})\overline{\alpha_h}\beta_g-(\overline{\alpha_e}\beta_e+\overline{\alpha_f}\beta_f)\alpha_g\overline{\beta_h}.
\end{align*}
Using that $a_1^*a_1=a_2^*a_2=\lambda$ and $a_1^*a_2=a_2^*a_1=0$ we get
\begin{align*}
d_{e,g}\overline{d_{e,h}}+d_{f,g}\overline{d_{f,h}}&=
(\lambda-|\alpha_g|^2-|\alpha_h|^2)\beta_g\overline{\beta_h}+(\lambda-|\beta_g|^2-|\beta_h|^2)\alpha_g\overline{\alpha_h}\\
&\qquad +(\alpha_g\overline{\beta_g}+\alpha_h\overline{\beta_h})\overline{\alpha_h}\beta_g+(\overline{\alpha_g}\beta_g+\overline{\alpha_h}\beta_h)\alpha_g\overline{\beta_h} \\ &=
\lambda(\beta_g\overline{\beta_h}+\alpha_g\overline{\alpha_h}).
\end{align*}
The proof is complete.
\end{proof}

\begin{proof}[Proof of Proposition \ref{d2yx}]
This follows combining Lemmas \ref{generatelemma} and \ref{defLemma}.
\end{proof}

\begin{proof}[Proof of Proposition \ref{sumsquared4}]
If $R=\mathbb Z[i]$ then by Proposition \ref{Orthoregbasisthm}(2) there exists an integral $Q'$-orthoregular basis of norm $|d_2(A_0)|^2$ (which is an absolute square in $R$). We assume hence $R=\ZZ$.

We may assume that $d_1(A_0)=1$, otherwise we can divide by the common divisor. That does not change $\Lambda$, $Q$, or 
$Q'$. By (1) of Proposition \ref{Orthoregbasisthm} we have to prove that $\alpha'$ (the upper left entry of the matrix $M'$) is a sum of two squares. By Lemma \ref{squaredetlem} it suffices to show that there is no prime $q=4k+3$ for which the exponent of $q$ in $Q'(x)$ is odd for all nonzero $x\in\mathbb Z^2$.

Assume for contradiction that there is such a prime $q$. By Lemma \ref{defLemma} we obtain
\[Q'(x_e)=\frac{d_2(A_0)^2}{\lambda}|x_e|^2=d_2(A_0)^2(\lambda-\alpha_e^2-\beta_e^2)\]
for all $1\leq e\leq 4$. As the exponent of $q$ in $d_2(A_0)^2$ is even we have $q\mid\lambda-\alpha_e^2-\beta_e^2$. Thus $q$ divides
\[\sum_{e=1}^4\left(\lambda-\alpha_e^2-\beta_e^2\right)=2\lambda,\]
whence it also divides both $\lambda$ and $\alpha_e^2+\beta_e^2$ for each $1\leq e\leq 4$. Since $q=4k+3$, we have $q\mid\alpha_e,\beta_e$
, hence $q\mid d_1(A_0)=1$, which is a contradiction.
\end{proof}

\subsection{The Smith normal form}\label{SNFsection}~

The proof of Propositions \ref{3Zip} and \ref{4Zip} will rely on the fact that if $A$ is an $n$-icube as in the statements, then $\SNF(A)$ is uniquely determined by (the argument of) $\det(A)$.

\begin{proof}[Proof of Proposition \ref{3Zip}]
We pick a basis $b_2,b_3$ of $\Lambda$ as in Proposition~\ref{discQlemma}.  
Using Proposition \ref{crossproductprop} we may assume $a_1=b_2\times b_3$ by possibly multiplying $b_3$ by a unit.

By Proposition \ref{Orthoregbasisthm} there is an integral  
a $Q$-orthoregular basis of norm $|a_1|^2=\Delta$ and type $\overline\alpha_2$, say $(x_2,y_2)^T,$ $(x_3,y_3)^T=(M(x_2,y_2)^T)_\perp/\overline\alpha_2\in R^2$. Then we obtain the corresponding icube $A=(a_1|a_2|a_3)$ by putting $a_2=x_2b_2+y_2b_3,a_3=x_3b_2+y_3b_3\in\Lambda$ 
and compute
\[a_2\times a_3=\det\left(\begin{pmatrix} x_2\\ y_2\end{pmatrix}\Bigg|\left(M\begin{pmatrix} x_2\\ y_2\end{pmatrix}\right)_\perp/\overline\alpha_2\right)b_2\times b_3=\frac{Q(x_2,y_2)}{\overline\alpha_2}a_1=\alpha_2 a_1 .\]
Hence we get $\det(A)=|a_1|^2\alpha_2=\Delta\alpha_2$. As $a_1$  
is primitive Lemma \ref{SNFlemma} finally yields $\SNF(A)=\diag(1,\alpha_2,\Delta)$.
\end{proof}

We shall need the following

\begin{lemma}\label{d2divlambda}
Let $A_0=(a_1|a_2)\in\mathbb Z[i]^{4\times 2}$ be an icube of norm $\lambda$. If $a_1$ is primitive, then $d_2(A_0)$ divides $\lambda$.
\end{lemma}
\begin{proof}
Define $\Lambda=Rb_3+Rb_4$ and $Q$ as in Proposition~\ref{discQlemma} and let $a$ be the cross product of $a_2,b_3$ and $b_4$ as in Proposition \ref{crossproductprop}. Put $B$ for the matrix $(a_1|a_2|b_3|b_4)$. Since $a_1$ is primitive, we may apply Proposition \ref{crossproductprop} to obtain $a=(\overline{\det(B)}/\lambda)\cdot a_1$. On the other hand, we have $\det(B)=\omega\lambda^2/\overline{d_2(A_0)}$ for some unit $\omega\in \ZZ[i]$ by \eqref{detBuptounit} which yields $a=\frac{\overline{\omega}\lambda}{d_2(A_0)}a_1$. Hence the greatest common divisor of the entries of $a$ is $\frac{\lambda}{d_2(A_0)}$ which lies in $\ZZ[i]$ since the entries of $a$ are integral.
\end{proof}

We finish the paper by the
\begin{proof}[Proof of Proposition \ref{4Zip}]
By Lemma \ref{d2divlambda} and the assumption that $(d_2(A_0),\overline{d_2(A_0)})=1$ the Smith normal form $\diag(1,\alpha_2,\lambda/\overline{\alpha_2},\lambda)$ makes sense for any divisor $\alpha_2$ of $d_2(A_0)$. Put $\delta:=\frac{d_2(A_0)\overline{\alpha_2}}{\alpha_2}\in \ZZ[i]$ so that $|\delta|^2=|d_2(A_0)|^2$. By (2) of Proposition \ref{Orthoregbasisthm} we know that there is an integral $Q'$-orthoregular basis $(x_3,y_3)^T, (x_4,y_4)^T$ of type $\delta$ which corresponds to an icube $A(\alpha_2)=(a_1|a_2|a_3|a_4)$ of norm $\lambda$ with $a_j=x_jb_3+y_jb_4$ ($j=3,4$). Using \eqref{detBuptounit} we get
\begin{align*}
\det(A(\alpha_2))=\det(a_1|a_2|b_3|b_4)\det \begin{pmatrix}x_3 & x_4 \\ y_3 & y_4\end{pmatrix}=\frac{\omega\lambda^2\overline\delta}{\overline{d_2(A_0)}}=\frac{\omega\lambda^2\alpha_2}{\overline{\alpha_2}}
\end{align*}
for some unit $\omega\in \ZZ[i]$.

Since $a_1$ is primitive, we have $d_1(A(\alpha_2))=1$ and $\alpha:=d_2(A(\alpha_2))$ divides $d_2(A_0)$ as $A_0$ is a submatrix of $A(\alpha_2)$. Therefore the Smith normal form of $A(\alpha_2)$ is of the form $\SNF(A(\alpha_2))=\diag(1,\alpha,\lambda/\overline\alpha,\lambda)$ for some $\alpha\mid d_2(A_0)$ by Lemma \ref{SNFlemma}. Comparing the determinants we find $
\alpha_2/\overline{\alpha_2}=\omega'\cdot\alpha/\overline{\alpha}
$ for some unit $\omega'\in \ZZ[i]$. Since the units of $\ZZ[i]$ are $\pm1,\pm i$ we deduce that the arguments of $\alpha$ and $\alpha_2$ differ by an integer multiple of $\frac{\pi}{4}$. Since the Gaussian integers with argument $\pm\frac{\pi}{4}$ are divisible by $1+i$ and $1+i\nmid d_2(A_0)$, by possibly replacing $\alpha$ by a unit multiple we may further assume that the arguments of $\alpha$ and $\alpha_2$ are equal. By the conditions on $d_2(A_0)$, its divisors are uniquely determined by their argument hence $\alpha=\alpha_2$. Thus $A(\alpha_2)$ is an icube with the required Smith normal form.
\end{proof}

\bibliographystyle{alpha}
\bibliography{biblio.bib}

\end{document}